\documentclass[12pt]{amsart}
\usepackage[parfill]{parskip}
\usepackage[toc,page]{appendix}
\usepackage{psfrag}
\usepackage{graphicx}
\usepackage{epstopdf}
\usepackage{xcolor}
\usepackage{pinlabel}
\usepackage{enumitem}
\usepackage{amsmath}
\usepackage{amssymb}
\usepackage{amscd}
\usepackage{picinpar}
\usepackage{times}
\usepackage{pb-diagram}
\usepackage{graphicx}
\usepackage{wrapfig}
\usepackage{xspace}
\usepackage{hyperref}
\usepackage{xcolor}
\usepackage{url}
\usepackage{float}

\newtheorem{theorem}{Theorem}[section]
\newtheorem{lemma}[theorem]{Lemma}
\newtheorem{proposition}[theorem]{Proposition}

\newtheorem{definition}[theorem]{Definition}

\newtheorem{remark}[theorem]{Remark}
\newtheorem*{theorem*}{Theorem}
\newtheorem{manualtheoreminner}{}
\newenvironment{manualtheorem}[1]{%
  \renewcommand\themanualtheoreminner{#1}%
  \manualtheoreminner
}{\endmanualtheoreminner}

\begin{document}

\title{The convergence of discrete uniformizations for closed surfaces}

\author{Tianqi Wu}
\address{Center of Mathematical Sciences and Applications, Harvard University, Cambridge, MA 02138
}
\email{tianqi@cmsa.fas.harvard.edu}

\author{Xiaoping Zhu}
\address{Department of Mathematics, Rutgers University, Piscataway, NJ, 08854}
\email{xz349@math.rutgers.edu}



\begin{abstract}
The notions of discrete conformality on triangle meshes have  rich mathematical theories and wide applications. The related notions of discrete uniformizations on triangle meshes, suggest efficient methods for computing the uniformizations of surfaces. This paper proves that the discrete uniformizations approximate the continuous uniformization for closed surfaces of genus $\geq1$, when the approximating triangle meshes are reasonably good.
To the best of the authors' knowledge, this is the first convergence result on computing uniformizations for surfaces of genus $>1$.
\end{abstract}
\maketitle
\tableofcontents

\newtheorem*{equicont}{Theorem \ref{equicont}}
\section{Introduction}
The celebrated Poincar\'e-Koebe uniformization theorem states that any simply connected Riemannian surface $(M,g)$ is conformally equivalent to the unit sphere $\mathbb S^2$, or the complex plane $\mathbb C$, or the open unit disk $\mathbb D$. As a consequence, any smooth Riemannian metric $g$ on a closed orientable surface $M$ is conformally equivalent to a Riemannian metric $\tilde g$ of constant curvature $0$ or $\pm1$.
Computing such a uniformization map to $\mathbb S^2$ or $\mathbb C$ or $\mathbb D$, or such a uniformization metric $\tilde g$,
has a wide range of applications in surface parameterizations, classifications, and matchings and so on. See \cite{koehl2013automatic}, \cite{gu2004genus}, \cite{boyer2011algorithms}, \cite{hong2006conformal}, \cite{lui2010optimized} for examples.

Most of the existing methods of computing uniformizations are only for topological disks or topological spheres. Luo \cite{luo2004combinatorial} and Bobenko-Pinkall-Springborn \cite{bobenko2015discrete} et al. developed the theory of discrete conformality on triangle meshes, based on the notion of vertex scalings. Here, discrete uniformizations for triangle meshes can be naturally defined, and computed efficiently by minimizing an explicit convex functional. So for surfaces approximated by triangle meshes, we can compute the (approximated) uniformizations efficiently.

This paper proves the convergence of this computing method, for smooth closed surfaces of genus $\geq1$. Briefly speaking, we proved that for a smooth Riemannian surface $M$ of genus $\geq1$ and a regular dense geodesic triangulation $T$ of $M$, the discrete uniformization of $T$ approximates the uniformization of $M$, with an error bounded linearly by the maximum size of the triangles in $T$. Our result is a strengthening and a generalization of Gu-Luo-Wu's convergence theorem (Theorem 6.1 in \cite{gu2019convergence}), which is only for the case of genus 1.
To the best of the authors' knowledge, this is the first convergence result on computing uniformizations for surfaces of genus $>1$.

Aside from other technicalities, the main ingredients of the proofs are
\begin{itemize}
\item a cubic estimate (Lemma \ref{cubic estimate}) first given by Gu-Luo-Wu \cite{gu2019convergence}, showing that the discrete conformal change approximates the continuous conformal change, and

\item explicit formulae (Proposition \ref{Discrete curvature equation for torus case} and \ref{Discrete curvature equation for hyperbolic case})  given by Bobenko-Pinkall-Springborn \cite{bobenko2015discrete}, showing that the differential of the discrete curvature map is indeed a discrete elliptic operator, and

\item a technical elliptic estimate on graphs (Lemma \ref{estimate for divergence operator}), based on discrete isoperimetric inequalities on triangle meshes.
\end{itemize}

\subsection{Set Up and Main Theorems}
\label{Notation and Statement}
Assume $M$ is a two dimensional closed orientable surface, and $T$ is a \emph{triangulation} of $M$, which is a simplicial complex. Denote $V(T),E(T),F(T)$ as
the set of vertices, edges, and triangles of $T$ respectively. Further for $M$ equipped with a Riemannian metric, $T$ is called a \emph{geodesic triangulation} if any edge in $T$ is a geodesic segment.
In this paper, a Riemannian surface $(M,g)$ is always assumed to be smooth, i.e., $C^\infty$.

Given a triangulation $T$, if an edge length $l\in\mathbb R^{E(T)}_{>0}$ satisfies the triangle inequalities, then we can determine a Euclidean triangle mesh $(T,l)_E$ by assuming that each triangle in $F(T)$ is a Euclidean triangle with the edge lengths given by $l$.
We can construct an analogous hyperbolic triangle mesh $(T,l)_H$, by using hyperbolic triangles instead of the Euclidean triangles.
Notice that a Euclidean (\emph{resp.} hyperbolic) triangle mesh has a piecewise Euclidean (\emph{resp.} hyperbolic) metric, and all the singular points (or cone points) are in $V(T)$.
Given $(T,l)_E$ or $(T,l)_H$, $\theta^i_{jk}$ denotes the inner angle at $i$ in the triangle $\triangle ijk\in F(T)$, and the \emph{discrete curvature} $K_i$ at a vertex $i\in V(T)$ is defined to be the angle defect
$$
K_i=2\pi-\sum_{jk:ijk\in F(T)}\theta^i_{jk}.
$$
It is easy to see that $(T,l)_E$ (\emph{resp.} $(T,l)_H$) is globally flat (\emph{resp.} globally hyperbolic) if and only if $K_i=0$ for any $i\in V$.

\begin{definition}[\cite{luo2004combinatorial}\cite{bobenko2015discrete}]\label{definition 1.1}
The Euclidean triangle meshes $(T,l)_E$ and $(T,l')_E$ are called \emph{discrete conformal }if there exists some $u\in\mathbb R^{V(T)}$, such that
\begin{equation}
\label{Euclidean discrete conformal}
l'_{ij}=e^{\frac{1}{2}(u_i+u_j)}l_{ij}
\end{equation}
for any $ij\in E(T)$.
The hyperbolic triangle meshes $(T,l)_H$ and $(T,l')_H$ are  called \emph{discrete conformal }
if there exists some $u\in\mathbb R^{V(T)}$, such that
\begin{equation}
\label{hyperbolic discrete conformal}
\sinh\frac{l'_{ij}}{2}=e^{\frac{1}{2}(u_i+u_j)}\sinh\frac{l_{ij}}{2}
\end{equation}
for any $ij\in E(T)$.
\end{definition}
Such a vector $u\in\mathbb R^V$ is called a \emph{discrete conformal factor}, and we denote $l'=u*l$ if (\ref{Euclidean discrete conformal}) holds, and $l'=u*_hl$ if (\ref{hyperbolic discrete conformal}) holds.
Given $(T,l)_E$ or $(T,l)_H$, $\theta^i_{jk}(u)$ and $K_i(u)$ denote the corresponding inner angle and
the discrete curvature respectively,
in $(T,u*l)_E$ or $(T,u*_hl)_H$.
Let $K(u)=[K_i(u)]_{i\in V}\in\mathbb R^V$. Given $(T,l)_E$ or $(T,l)_H$, $u\in\mathbb R^V$ is called a \emph{discrete uniformization conformal factor} if $K(u)=0$, i.e., $(T,u*l)_E$ is globally flat or $(T,u*_hl)_H$ is globally hyperbolic.

Luo \cite{luo2004combinatorial} and Bobenko-Pinkall-Springborn \cite{bobenko2015discrete} developed the variational principles, saying that
$$
\mathcal F(u)=\int^u_{u^*} K(\tilde u)d\tilde u
$$
is well defined and locally convex. Here $\mathcal F(u)$ is only defined on a domain of $\mathbb R^V$ where $u*l$ or $u*_hl$ satisfies the triangle inequalities. Bobenko-Pinkall-Sprinborn \cite{bobenko2015discrete} gave simple and explicit formulae extending $\mathcal F$ to a globally convex functional $\tilde{\mathcal F}$ on $\mathbb R^V$. So if a discrete uniformization conformal factor exists, it can be computed efficiently by minimizing an explicit globally convex functional $\tilde {\mathcal F}$ on $\mathbb R^V$.

Given a smooth Riemannian surface $(M,g)$ of genus 1 (\emph{resp. genus $>1$}), and a geodesic triangulation $T$ with the edge length $l\in\mathbb R^{E(T)}_{>0}$ measured in $(M,g)$, we view $(T,l)_E$ (\emph{resp. }$(T,l)_H$) as a triangle mesh approximation of $(M,g)$ and prove that the discrete uniformization approximates the true uniformization of $(M,g)$ if the triangle mesh is sufficiently regular and dense. For the regularity of a triangle mesh, we use the following definition.
\begin{definition}\label{epsilon regular mesh}
 $(T,l)_E$ or $(T,l)_H$ is called \emph{$(\epsilon_1,\epsilon_2)$-regular} if
 \begin{enumerate}[label=(\alph*)]
	\item $\theta^i_{jk}\geq\epsilon_1$ for any inner angle, and
	\item $\theta^k_{ij}+\theta^{k'}_{ij}\leq\pi-\epsilon_2$ for any pair of adjacent triangles $\triangle ijk$ and $\triangle ijk'$.
\end{enumerate}
\end{definition}
Here part (b) is a uniformly strict Delaunay condition, and could be satisfied by, for example, uniformly acute triangle meshes.
Our main convergence results are the following Theorem \ref{high genus main thm} and \ref{torus main thm}.
In this paper, if $x\in\mathbb R^A$ is a vector for some finite set $A$, we use
$|x|$ to denote the infinite norm of $x$, i.e., $|x|=|x|_\infty=\max_{i\in A}|x_i|$.
\begin{theorem}\label{high genus main thm}
Suppose $(M,g)$ is a closed orientable smooth Riemannian surface with genus $>1$, and $\bar u=\bar u_{M,g}\in C^{\infty}(M)$ is the unique uniformization conformal factor such that $e^{2\bar u}g$ is hyperbolic.
Assume $T$ is a geodesic triangulation, and $l\in\mathbb R^{E(T)}_{>0}$ denotes the edge length in $(M,g)$.
Then for any $\epsilon_1,\epsilon_2>0$, there exists a constant $\delta=\delta(M,g,\epsilon_1,\epsilon_2)>0$  such that if $(T,l)_H$ is $(\epsilon_1,\epsilon_2)$-regular and $|l|<\delta$, then
\begin{enumerate}[label=(\alph*)]
\item there exists a unique discrete conformal factor $u\in V(T)$, such that $(T,u*_hl)_H$ is globally hyperbolic, and
\item$\big|u-\bar u|_{V(T)}\big|\leq C|l|$ for some $C=C(M,g,\epsilon_1,\epsilon_2)>0$.
\end{enumerate}
\end{theorem}

\begin{theorem}\label{torus main thm}
Suppose $(M,g)$ is a closed orientable smooth Riemannian surface of genus 1, and $\bar u=\bar u_{M,g}\in C^{\infty}(M)$ is the unique uniformization conformal factor such that $e^{2\bar u}g$ is flat and $Area(M,e^{2\bar u}g)=1$. Assume $T$ is a geodesic triangulation, and $l\in\mathbb R^{E(T)}_{>0}$ denotes the edge length in $(M,g)$.
Then for any $\epsilon_1,\epsilon_2>0$, there exists a constant $\delta=\delta(M,g,\epsilon_1,\epsilon_2)>0$
 such that if $(T,l)_E$ is $(\epsilon_1,\epsilon_2)$-regular and $|l|<\delta$, then
\begin{enumerate}[label=(\alph*)]
	\item there exists a unique discrete conformal factor $u\in\mathbb R^{V(T)}$, such that $(T,u*l)_E$ is globally flat and $Area((T,u*l)_E)=1$, and
	\item $\big|u-\bar u|_{V(T)}\big|\leq C|l|$ for some constant
 $C=C(M,g,\epsilon_1,\epsilon_2)$.
\end{enumerate}
\end{theorem}
\subsection{Outline of the Proofs}
For simplicity we use $\bar u$ to denote $\bar u|_{V(T)}$.
For both Theorem \ref{high genus main thm} and \ref{torus main thm}, we first show that $\bar u$ is already a good candidate for the discrete uniformization conformal factor, in the sense that $K(\bar u)$ is "very small", or equivalently, $(T,\bar u*l)_E$ (\emph{resp. }$(T,\bar u*_hl)_H$) is very close to be globally flat (\emph{resp. } globally hyperbolic).
Then by constructing a flow on the triangle mesh, we perturb $\bar u$ to $\bar u+\delta u$, such that $K(\bar u+\delta u)$ is exactly 0, i.e., $\bar u+\delta u$ is a discrete uniformization conformal factor.
Based on the fact that $\partial K/\partial u$ is indeed a discrete elliptic operator, $\delta u$ is shown to be bounded by $C|l|$ via a technical discrete elliptic estimate.
\subsection{Other Notations}
In the remaining part of this paper, if a geodesic triangle $\triangle ABC$ is given, we always denote
\begin{enumerate}[label=(\alph*)]
\item $a,b,c$ as the lengths of the edges opposite to $A,B,C$ respectively, and

\item $A,B,C$ as the inner angles, and

\item $|\triangle ABC|$ as the area of $\triangle ABC$.
\end{enumerate}

For a domain $U$ in a 2-dimensional surface, the area of $U$ is also denoted as $|U|$, or $|U|_g$ if it is induced from a Riemannian metric $g$. For a finite union $\gamma$ of curves, we denote its length as $s(\gamma)$, or $s_g(\gamma)$ if it is induced from a Riemannian metric $g$.
For a point $x$ on the Riemannian surface $(M,g)$ and a radius $r>0$, denote
$$
B(x,r)=B_g(x,r)=\{y\in M:d_g(x,y)<r\}.
$$
A triangle mesh $(T,l)_E$ or $(T,l)_H$ is simply called \emph{$\epsilon$-regular} if it is $(\epsilon,\epsilon)$-regular.

\subsection{Organization of the Paper}
In Section \ref{section 2} we set up basic notions and properties for discrete calculus on graphs. Section \ref{differential of discrete curvatures and angles} reviews explicit formulae for $\partial K/\partial u$, and introduces elementary estimates on triangles.
Theorem \ref{high genus main thm} and \ref{torus main thm} are proved in Section \ref{section 4}, assuming three elementary geometric lemmas and a discrete elliptic estimate. Section \ref{Proof of geometric lemmas} proves the auxiliary lemmas and Section \ref{proof of the key estimates on graphs} proves the discrete elliptic estimate.

\subsection{Acknowledgement}
The work is supported in part by NSF 1760471, NSF DMS 1737876, NSF 1760527, and NSF 1811878.

\section{Calculus on Graphs}
\label{section 2}
Assume $G=(V,E)$ is an undirected connected simple graph, on which we will frequently consider vectors in $\mathbb R^V$, $\mathbb R^E$ and $\mathbb R^E_A$. Here $\mathbb R^E$ and $\mathbb R^E_A$ are both vector spaces of dimension $|E|$ such that

\begin{enumerate}[label=(\alph*)]
\item a vector $x\in \mathbb R^E$ is represented symmetrically, i.e., $x_{ij}=x_{ji}$, and
\item a vector $x\in\mathbb R^E_A$ is represented anti-symmetrically, i.e., $x_{ij}=-x_{ji}$.
\end{enumerate}
A vector in $\mathbb R^E_A$ is also called a \emph{flow} on $G$.
An \emph{edge weight} $\eta$ on $G$ is a vector in $\mathbb R^E$. Given an edge weight $\eta$, the \emph{gradient} $\nabla x=\nabla_\eta x$ of a vector $x\in\mathbb R^V$ is a flow in $\mathbb R^E_A$ such that
$$
(\nabla x)_{ij}=\eta_{ij}(x_j-x_i).
$$
Given a flow $x\in\mathbb R^E_A$, its \emph{divergence} $div(x)$ is a vector in $\mathbb R^V$ such that
$$
div(x)_i=\sum_{j\sim i}x_{ij}.
$$
Given an edge weight $\eta$, the associated \emph{Laplacian} $\Delta=\Delta_\eta:\mathbb R^V\rightarrow\mathbb R^V$ is defined as
$\Delta x=div(\nabla_\eta x)$, i.e.,
$$
(\Delta x)_i=\sum_{j\sim i}\eta_{ij}(x_j-x_i).
$$
There is a discrete Green's identity on graphs.
\begin{proposition}[Green's identity]
	\label{green's identity}
Given $x,y\in\mathbb R^V$,
$$
\sum_{i\in V}x_i(\Delta y)_i=\sum_{i\in V}y_i(\Delta x)_i.
$$
\end{proposition}
\begin{proof}
$$
\sum_{i\in V}x_i(\Delta y)_i=\sum_{i\in V}x_i\sum_{j\sim i}\eta_{ij}(y_j-y_i)=\sum_{ij\in E}\eta_{ij}x_iy_j-\sum_{i\in V}x_i y_i\sum_{j\sim i}\eta_{ij}.
$$
Then by symmetry the Green's identity holds.
\end{proof}
A Laplacian is a linear transformation on $\mathbb R^V$, and could be identified as a $|V|\times |V|$ symmetric matrix.
By the definition, $\Delta\bold 1=0$ where $\bold 1=(1,1,...,1)\in\mathbb R^V$. Also, it is  well known that if $\eta\in\mathbb R^E_{>0}$, then $ker(\Delta)=\mathbb R\bold 1$ by the connectedness of the graph $G$.

In the rest of this section, we always assume $\eta\in\mathbb R^E_{>0}$.
So $\Delta$ is invertible on the subspace $\bold 1^\perp=\{x\in\mathbb R^V:\sum_{i\in V}x_i=0\}$.
Denote $\Delta^{-1}$ as the inverse of $\Delta$ on $\bold 1^\perp$. The following regularity property will be needed in the proof of our main theorem.
\begin{lemma}
	\label{smoothness of operator in hyperbolic case}
$(\eta,y)\mapsto\Delta_\eta^{-1}y$ is a smooth map from $\mathbb R^V_{>0}\times \bold 1^{\perp}$ to $\bold 1^\perp$.
\end{lemma}
\begin{proof}
It is equivalent to show that $\Phi:(\eta,y)\mapsto(\eta,\Delta_\eta^{-1}y)$ is a smooth mapping from $\mathbb R^V_{>0}\times\bold 1^\perp$ to itself.
By the inverse function theorem, it suffices to show that $\Phi^{-1}(\eta,x)=(\eta,\Delta_\eta x)$ is smooth and $D(\Phi^{-1})$ is non-degenerate. The smoothness is obvious, and
$$
D(\Phi^{-1})=
\begin{pmatrix}
id& \partial \eta/\partial x\\
\partial (\Delta_\eta x)/\partial \eta&\partial (\Delta_\eta x)/\partial x
\end{pmatrix}
=
\begin{pmatrix}
id& 0\\
\partial (\Delta_\eta x)/\partial \eta&\Delta_\eta
\end{pmatrix}
$$
is indeed nondegenerate, since $\Delta_\eta$ is invertible on $\bold 1^{\perp}$.
\end{proof}
Now we introduce the notion of $C$-isoperimetry for a graph $G=(V,E)$ associated with a positive vector $l\in\mathbb R^E_{>0}$.
Given any $V_0\subset V$, denote
$$
\partial V_0=\{ij\in E:i\in V_0,j\notin V_0\},
$$
and then define the $l$-\emph{perimeter} of $V_0$ and the $l$-\emph{area} of $V_0$ as
\begin{equation*}
	|\partial V_0|_l=\sum_{ij\in \partial V_0}l_{ij}\quad\text{ and }\quad
	|V_0|_l=\sum_{i,j\in V_0,ij\in E}l_{ij}^2
\end{equation*}
respectively.

For a constant $C>0$,
such a pair $(G,l)$ is called $C$-\emph{isoperimetric} if for any $V_0\subset V$
\begin{equation*}
\min\{|V_0|_l,|V|_l-|V_0|_l\}\leq C\cdot|\partial V_0|_l^2.
\end{equation*}
We will see, from part (b) of Lemma \ref{isoperimetric}, that a uniform $C$-isoperimetric condition is satisfied by regular triangle meshes approximating a closed smooth surface.
The following discrete elliptic estimate plays an important role in proving our main theorems. The techinical proof is postponed to Section \ref{proof of the key estimates on graphs}.
\begin{lemma}
	\label{estimate for divergence operator}
Assume $(G,l)$ is $C_1$-isoperimetric, and $x\in\mathbb R^E_A,\eta\in\mathbb R^E_{>0},C_2>0,C_3>0$ are such that
\begin{enumerate}[label=(\roman*)]
\item $|x_{ij}|\leq C_2 l_{ij}^2$ for any $ij\in E$, and

\item $\eta_{ij}\geq C_3$ for any  $ij\in E$.
\end{enumerate}
Then
\begin{equation*}
|\Delta^{-1}_\eta\circ div (x))|\leq \frac{4C_2\sqrt{C_1+1}}{C_3}|l|\cdot|V|_l^{1/2}.
\end{equation*}
Further if $y\in\mathbb R^V$ and $C_4>0$ and $D\in\mathbb R^{V\times V}$ is a diagonal matrix such that
$$
|y_i|< C_4 D_{ii}|l|\cdot|V|_l^{1/2}
$$
for any $i\in V$, then
\begin{equation*}
|(D-\Delta_\eta)^{-1} (div(x)+y)|\leq
\left(C_4+\frac{8C_2\sqrt{C_1+1}}{C_3}\right)|l|\cdot|V|_l^{1/2}.
\end{equation*}
\end{lemma}

\section{Differential of the Discrete Curvatures and Angles}
\label{differential of discrete curvatures and angles}
Bobenko-Pinkall-Springborn \cite{bobenko2015discrete} gave explicit formulae
for the infinitesimal changes of the discrete curvature, as the mesh deform in its discrete conformal equivalence class. Here we reformulate their formulae as follows.
\begin{proposition}
[Proposition 4.1.6 in \cite{bobenko2015discrete}]
\label{Discrete curvature equation for torus case}
Given $(T,l)_E$ and $u\in\mathbb R^{V(T)}$ such that $u*l$ satisfies the triangle inequalities, define the cotangent weight $\eta\in\mathbb R^E$ as
$$
\eta_{ij}(u)=\frac{1}{2}\cot\theta^k_{ij}(u)+\frac{1}{2}\cot\theta^{k'}_{ij}(u)
$$
where $\triangle ijk$ and $\triangle ijk'$ are adjacent triangles in $F(T)$. Then
$$
\frac{\partial K}{\partial u}(u)=-\Delta_{\eta(u)}.
$$
\end{proposition}
\begin{proposition}[Proposition 6.1.7 in \cite{bobenko2015discrete}]
	\label{Discrete curvature equation for hyperbolic case}
Given $(T,l)_H$ and $u\in\mathbb R^{V(T)}$  such that $u*_hl$ satisfies the triangle inequalities,
denote
$$
\tilde\theta^i_{jk}(u)=\frac{1}{2}(\pi+\theta^i_{jk}(u)-\theta^j_{ik}(u)-\theta^k_{ij}(u))
$$
and
$$
w_{ij}(u)=\frac{1}{2}\cot\tilde\theta^k_{ij}(u)+
\frac{1}{2}\cot\tilde\theta^{k'}_{ij}(u)
$$
where $\triangle ijk$ and $\triangle ijk'$ are adjacent triangles in $F(T)$.
Then
\begin{equation*}
\frac{\partial K}{\partial u}(u)=D(u)-\Delta_{\eta(u)}
\end{equation*}
where
$$
\eta_{ij}(u)=w_{ij}(u)(1-\tanh^2\frac{(u*_h l)_{ij}}{2}),
$$
and
$D=D(u)$ is a diagonal matrix such that
$$
D_{ii}(u)=2\sum_{j:ij\in E}w_{ij}(u)\tanh^2\frac{(u*_h l)_{ij}}{2}.
$$
\end{proposition}
To derive Proposition \ref{Discrete curvature equation for torus case} and  \ref{Discrete curvature equation for hyperbolic case}, one only needs to compute for a single triangle as in Lemma \ref{Euclidean infinitesimal} and \ref{hyperbolic infinitesimal}, and then properly add up the following equation (\ref{Euclidean differential for conformal factor}) for the Euclidean case, and
(\ref{hyperbolic differential for conformal factor})(\ref{hyperbolic differential for conformal factor2}) for the hyperbolic case. We omit the proofs of Proposition \ref{Discrete curvature equation for torus case} and \ref{Discrete curvature equation for hyperbolic case}, and postpone the elementary calculations for
Lemma \ref{Euclidean infinitesimal} and \ref{hyperbolic infinitesimal} to Appendix.

\begin{lemma}\label{Euclidean infinitesimal}
Given a Euclidean triangle $\triangle ABC$, if we view $A,B,C$ as functions of the edge lengths $a,b,c$, then
 \begin{equation*}
 \label{Euclidean differential}
 \frac{\partial A}{\partial b}=-\frac{\cot C}{b},\quad\quad
 \frac{\partial A}{\partial a}=\frac{\cot B+\cot C}{a}=\frac{1}{b\sin C}.
 \end{equation*}

Further if $(u_A,u_B,u_C)\in\mathbb R^3$ is a discrete conformal factor, and
\begin{equation*}
a=e^{\frac{1}{2}(u_B+u_C)}a_0,\quad b=e^{\frac{1}{2}(u_A+u_C)}b_0,
\quad c=e^{\frac{1}{2}(u_A+u_B)}c_0
\end{equation*}
for some constants $a_0,b_0,c_0\in\mathbb R_{>0}$, then
\begin{equation}
\label{Euclidean differential for conformal factor}
\frac{\partial A}{\partial u_B}=\frac{1}{2}\cot C,\quad
\frac{\partial A}{\partial u_A}=-\frac{1}{2}(\cot B+\cot C).
\end{equation}
\end{lemma}

\begin{lemma}\label{hyperbolic infinitesimal}
Given a hyperbolic triangle $\triangle ABC$, if we view $A,B,C$ as functions of the edge lengths $a,b,c$, then
 \begin{equation*}
 \label{hyperbolic differential}
 \frac{\partial A}{\partial b}=-\frac{\cot C}{\sinh b},\quad\quad
 \frac{\partial A}{\partial a}=\frac{1}{\sinh b\sin C}.
 \end{equation*}

Further if $(u_A,u_B,u_C)\in\mathbb R^3$ is a discrete conformal factor, and
\begin{equation*}
\sinh\frac{a}{2}=e^{\frac{1}{2}(u_B+u_C)}\sinh\frac{a_0}{2},\quad \sinh\frac{b}{2}=e^{\frac{1}{2}(u_A+u_C)}\sinh\frac{b_0}{2},
\quad \sinh\frac{c}{2}=e^{\frac{1}{2}(u_A+u_B)}\sinh\frac{c_0}{2}
\end{equation*}
for some constants $a_0,b_0,c_0\in\mathbb R_{>0}$, then
\begin{equation}
\label{hyperbolic differential for conformal factor}
\frac{\partial A}{\partial u_B}=\frac{1}{2}\cot \tilde C(1-\tanh^2\frac{c}{2}),
\end{equation}
and
\begin{equation}
\label{hyperbolic differential for conformal factor2}
\frac{\partial A}{\partial u_A}
=-\frac{1}{2}\cot \tilde B(1+\tanh^2\frac{b}{2})-\frac{1}{2}\cot \tilde C(1+\tanh^2\frac{c}{2}),
\end{equation}
where
$
\tilde B=\frac{1}{2}(\pi+B-A-C)
$
and
$
\tilde C=\frac{1}{2}(\pi+C-A-B).
$
\end{lemma}
By the differential formulae in Lemma \ref{Euclidean differential} and \ref{hyperbolic differential},
it is not difficult to prove the following estimates Lemma \ref{estimate for Euclidean triangle} and \ref{estimate for hyperbolic triangle} for perturbations of a single triangle.
The proofs of
Lemma \ref{estimate for Euclidean triangle} and \ref{estimate for hyperbolic triangle}
are also in the Appendix.

\begin{lemma}\label{estimate for Euclidean triangle}
Given a Euclidean triangle $\triangle ABC$, if all the angles in $\triangle ABC$ are at least $\epsilon>0$, and $\delta<\epsilon^2/48$, and
$$
|a'-a|\leq \delta a,\quad |b'-b|\leq \delta a,\quad |c'-c|\leq \delta c,
$$
then $a',b',c'$ form a Euclidean triangle with opposite inner angles $A',B',C'$ respectively, and
$$
|A'-A|\leq\frac{24}{\epsilon}\delta,
$$
and
$$
\bigg||\triangle A'B'C'|-|\triangle ABC|\bigg|\leq \frac{576}{\epsilon^2}\delta\cdot |\triangle ABC|.
$$
\end{lemma}

\begin{lemma}\label{estimate for hyperbolic triangle}
Given a hyperbolic triangle $\triangle ABC$, if all the angles in $\triangle ABC$ are at least $\epsilon>0$, and $\delta<\epsilon^3/60$, and
$$
a\leq0.1,\quad b\leq0.1
,\quad c\leq0.1,
$$
and
$$
|a'-a|\leq \delta a,\quad |b'-b|\leq \delta a,\quad |c'-c|\leq \delta c,
$$
then $a',b',c'$ form a hyperbolic triangle with opposite inner angles $A',B',C'$ respectively, and
$$
|A'-A|\leq\frac{30}{\epsilon^2}\delta,
$$
and
$$
\bigg||\triangle A'B'C'|-|\triangle ABC|\bigg|\leq \frac{120}{\epsilon^2}\delta \cdot |\triangle ABC|.
$$
\end{lemma}

\section{Proof of the Main Theorems}
\label{section 4}
In this section we first introduce three geometric lemmas, whose proofs are given in Section \ref{Proof of geometric lemmas}, and then prove Theorem \ref{torus main thm} and \ref{high genus main thm} in Subsection \ref{the case of torus} and \ref{the case of high genus} respectively.
\subsection{Geometric Lemmas}
\begin{lemma}
\label{triangle area}
Suppose $\triangle_E ABC$, $\triangle_H ABC$ and $\triangle_S ABC$ are Euclidean and hyperbolic and spherical triangles respectively, with the same edge lengths $a,b,c<0.1$.

(a)
If all the inner angles in $\triangle_E ABC$ are at least $\epsilon>0$, then for any $P\in\{E,H,S\}$,
$$
\frac{\epsilon}{8} a^2\leq |\triangle_P ABC|\leq \frac{1}{\epsilon}a^2.
$$

(b)
Assume $M_a$ is the middle point of $BC$, and $M_b$ is the middle point of $AC$, and $\triangle_P CM_aM_b$ is the geodesic triangle in $\triangle_PABC$ with vertiecs $C,M_a,M_b$, where $P\in\{E,H,S\}$.
Then
$$
|\triangle_P CM_aM_b|\geq \frac{1}{5}|\triangle_P ABC|
$$
for any $P\in\{E,H,S\}$.
\end{lemma}
\begin{remark}
\label{remark}
By the well-known Toponogov comparison theorem (see Lemma \ref{Toponogov}),
the assumption in part (a) of Lemma \ref{triangle area} can be replaced by that all the inner angles in $\triangle_HABC$ are at least $\epsilon>0$.
\end{remark}

\begin{figure}
	\label{midpoint}
	\centering
	\includegraphics[width=0.5\textwidth]{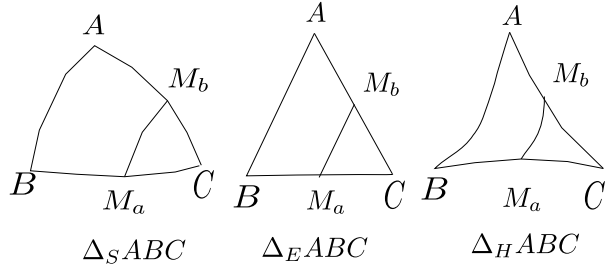}
	\caption{Triangles in Lemma \ref{triangle area}}
\end{figure}
The following Lemma \ref{cubic estimate} was first proved by Gu-Luo-Wu (see Proposition 5.2 in  \cite{gu2019convergence}). It indicates that our discrete conformal formula $l'_{ij}=e^{(u_i+u_j)/2}l_{ij}$ accurately approximates the local distance after a continuous conformal change.
\begin{lemma}
	\label{cubic estimate}
Suppose $(M,g)$ is a  closed Riemannian surface, and $ u\in C^{\infty}(M)$ is a conformal factor. Then there exists $C=C(M,g,u)>0$ such that for any $x,y\in M$,
$$
|d_{e^{2 u}g}(x,y)-e^{\frac{1}{2}(u(x)+ u(y))}d_g(x,y)|\leq Cd_g(x,y)^3.
$$
\end{lemma}

The following Lemma \ref{isoperimetric} shows that a regular dense geodesic triangulation $(T,l)$ of a surface is $C$-isoperimetric and "stable" under a conformal change.
\begin{lemma}
	\label{isoperimetric}
Suppose $(M,g)$ is a closed Riemannian surface,
and $T$ is a geodesic triangulation of $(M,g)$, and $l\in\mathbb R_{>0}^{E(T)}$ denotes the geodesic lengths of the edges, and
assume $(T,l)_E$ or $(T,l)_H$ is $\epsilon$-regular.
\begin{enumerate}[label=(\alph*)]

\item Given a conformal factor $ u\in C^{\infty}(M)$, there exists a constant $\delta=\delta(M,g,u,\epsilon)>0$ such that if $|l|<\delta$ then
there exists a geodesic triangulation $T'$ in $(M,e^{2u}g)$ such that $V(T')=V(T)$, and $T'$ is homotopic to $T$ relative to $V(T)$. Further $(T,\bar l)_E$ and $(T,\bar l)_H$ are $\frac{1}{2}\epsilon$-regular where
$\bar l\in\mathbb R^{E(T)}\cong\mathbb R^{E(T')}$ denotes the geodesic lengths of the edges of $T'$ in $(M,e^{2u}g)$.
\item There exists a constant $\delta=\delta(M,g,\epsilon)$ such that if $|l|<\delta$,
$(T,l)$ is $C$-isoperimetric for some constant $C=C(M,g,\epsilon)$.
\end{enumerate}
\end{lemma}

\subsection{Proof of Theorem \ref{torus main thm}}
\label{the case of torus}
For fix two constants $\epsilon_1,\epsilon_2>0$, we assume that $(T,l)_E$ is $(\epsilon_1,\epsilon_2)$-regular and $|l|\leq \delta$ where $\delta=\delta(M,g,\epsilon_1,\epsilon_2)<1$ is a sufficiently small constant to be determined. Simply note that $(T,l)_E$ is $\epsilon$-regular where $\epsilon=\min\{\epsilon_1,\epsilon_2\}$.
By Lemma \ref{isoperimetric},
if $\delta$ is sufficiently small then there exists a geodesic triangulation $T'$ of $(M,e^{2\bar u}g)$ homotopic to $T$ relative to $V(T)=V(T')$.
Let $\bar l\in\mathbb R^{E(T)}\cong \mathbb R^{E(T')}$ denote the geodesic lengths of the edges of $T'$ in $(M,e^{2\bar u}g)$, and then $(T,\bar l)_E$ is isometric to $(M,e^{2\bar u}g)$ and globally flat.

For simplicity, we will frequently use the notion $a=O(b)$ to denote that if $\delta=\delta(M,g,\epsilon_1,\epsilon_2)$ is sufficiently small, then $|a|\leq C\cdot b$ for some constant $C=C(M,g,\epsilon_1,\epsilon_2)$. For example, $l_{ij}=O(l_{jk})$ for any $\triangle ijk\in F(T)$, and $(\bar u*l)_{ij}=O(l_{ij})$, and $\bar l_{ij}=O(l_{ij})$.
The remaining of the proof is divided into three steps.
\begin{enumerate}[]
\item Firstly we show that $(T,\bar u*l)_E$ is very close to the globally flat triangle mesh $(T,\bar l)_E$, in the sense that
$$
(\bar u*l)_{ij}-\bar l_{ij}=O(l_{ij}^3)
$$
and
$$
K(\bar u)=div(x)
$$
for some flow $x\in\mathbb R^E_A$ satisfying $x_{ij}=O(l_{ij}^2).$

\item Secondly, we construct a flow $u(t):[0,1]\rightarrow\mathbb R^V$, starting at $u(0)=\bar u$, to linearly eliminate the curvature $K(\bar u)$, i.e., to let
    $$
    K(u(t))=(1-t)K(\bar u).
    $$ Further we show that $|u'(t)|=O(|l|)$, and then $(T,u(1)*l)_E$ is globally flat and $u(1)-\bar u=O(|l|)$.

\item Lastly we show that $Area((T,u(1)*l)_E)-1=O(|l|)$, so the area normalization condition can be satisfied by slightly scaling $(T,u(1)*l)_E$.
\end{enumerate}

The uniqueness of the discrete uniformization conformal factor is proved by Bobenko-Pinkall-Springborn (see Theorem 3.1.4 in \cite{bobenko2015discrete}), so we omit its proof here.

\subsubsection{Step 1}
\label{torus 1}
By lemma \ref{isoperimetric}, $(T,\bar{l})_{E}$ is $\frac{1}{2}\epsilon$-regular if $\delta$ is sufficiently small. For simplicity we denote $\bar u|_{V(T)}$ as $\bar u$.
By lemma \ref{cubic estimate},
\begin{equation*}
\label{11}
\bar l_{ij}-(\bar u*l)_{ij}=O (l_{ij}^3),
\end{equation*}
and then by Lemma \ref{estimate for Euclidean triangle}
$$
\alpha^i_{jk}:=\bar\theta^i_{jk}-\theta^i_{jk}(\bar u)=O(l_{ij}^2)
$$
where $\bar\theta^i_{jk}$ denotes the inner angle in $(T,\bar l)_E$.
So $(T,\bar u* l)_E$ is $\frac{1}{3}\epsilon$-regular if $\delta$ is sufficiently small.
Let $x\in\mathbb R^E_A$ be such that
$$
x_{ij}=\frac{\alpha^i_{jk}-\alpha^j_{ik}}{3}+\frac{\alpha^i_{jk'}-\alpha^j_{ik'}}{3}
$$
where $\triangle ijk$ and $\triangle ijk'$ are adjacent triangles. Then
$\alpha^i_{jk}+\alpha^j_{ik}+\alpha^k_{ij}=0$ and
$$
div(x)_i=\sum_{j:j\sim i}x_{ij}=
\sum_{jk:\triangle ijk\in F(T)}
\left(\frac{\alpha^i_{jk}-\alpha^j_{ik}}{3}+\frac{\alpha^i_{jk}-\alpha^k_{ij}}{3}\right)
=\sum_{jk:\triangle ijk\in F(T)}\alpha^i_{ij}=K_i(\bar u),
$$
and
\begin{equation}
\label{13}
x_{ij}=O(l_{ij}^2).
\end{equation}
\subsubsection{Step 2}
Let
$$
\tilde\Omega=\{u\in\bold 1^\perp:u*l\text{ satisfies the triangle inequalities and $(T,u*l)_E$ is $\frac{\epsilon}{5}$-regular}\}
$$
and
$$
\Omega=\{u\in\tilde\Omega:|u-\bar u|\leq1\text{, $(T,u*l)_E$ is $\frac{\epsilon}{4}$-regular}\}.
$$
Since $(T,\bar u*l)_E$ is $\frac{1}{3}\epsilon$-regular,
$\bar u$ is in the interior of $\Omega$.
Now consider the following ODE on $int(\tilde\Omega)$,
\begin{equation}
\label{ode for torus case}
	\left\{
	\begin{array}{lcl}
	u'(t)=\Delta_{\eta(u)}^{-1} K(\bar u)=\Delta_{\eta(u)}^{-1}\circ div(x)\\
	u(0)=\bar u
	\end{array}
	\right.,
\end{equation}
where
\begin{equation}
\label{15}
	\eta_{ij}(u)=\frac{\cot\theta^k_{ij}(u)+\cot\theta^{k'}_{ij}(u)}{2}
	=\frac{\sin(\theta^k_{ij}(u)+\theta^{k'}_{ij}(u))}{2\sin\theta^k_{ij}(u)\sin\theta^{k'}_{ij}(u)}
	\geq\frac{1}{2}\sin(\theta^k_{ij}(u)+\theta^{k'}_{ij}(u))
	\geq\frac{1}{2}\sin\frac{\epsilon}{5},
\end{equation}

where $\triangle ijk,\triangle ijk'$ are adjacent triangles.
By lemma \ref{smoothness of operator in hyperbolic case}, the right-hand side of (\ref{ode for torus case}) is a smooth function of $u$, so the ODE (\ref{ode for torus case}) has a unique solution $u(t)$ and
$$
K(u(t))=(1-t)K(\bar u)
$$
by Proposition \ref{Discrete curvature equation for torus case}.
Assume the maximum existing open interval of $u(t)$ in $\Omega$ is $[0,T_0)$ where $T_0\in(0,\infty]$.
By Lemma \ref{isoperimetric}, $(T,l)$ is $C$-isoperimetric for some constant $C=C(M,g,\epsilon_1,\epsilon_2)$. Then for any $u\in\Omega$, $(T,u*l)$ is $(e^{4(|\bar u|+1)}C)$-isoperimetric by the fact that $|u|\leq|\bar u|+1$.
Then by Lemma \ref{estimate for divergence operator} and equations (\ref{13})(\ref{ode for torus case})(\ref{15}), for any $t\in[0,T_0)$
\begin{equation}
\label{u' for torus}
|u'(t)| =O(|l| \cdot|V|_l^{1/2}).
\end{equation}
By Lemma \ref{triangle area}
and
the fact that $(T,\bar l)_E$ is $\frac{1}{2}\epsilon$-regular,
\begin{align}
\label{V for torus}
&|V|_{l}=\sum_{ij\in E}l_{ij}^2=O(\sum_{ij\in E}\bar l_{ij}^{~2})
=O(\sum_{ijk\in F}(\bar l_{ij}^{~2}+\bar l_{jk}^{~2}+\bar l_{ik}^{~2}))\\
=&O(
\sum_{ijk\in F}|(\triangle ijk,\bar l)_E|)=O(|(T,\bar l)_E|)=O(1).\notag
\end{align}
Here recall that $|(\triangle ijk,\bar l)_E|$ denotes the area of the Euclidean triangle, and $|(T,\bar l)_E|$ denotes the area of the piecewise flat surface.

Combining the estimates (\ref{u' for torus}) and (\ref{V for torus}), we have that for any $t\in[0,T_0)$
\begin{equation*}
|u'(t)|=O(|l|).
\end{equation*}

If $T_0<1$, by Lemma \ref{estimate for Euclidean triangle}
\begin{equation*}
|u(T_0)-\bar u|=O(|l|)\quad\text{ and }\quad\theta^i_{jk}(u(T_0))-\theta^i_{jk}(\bar u)=O(|l|),
\end{equation*}
and thus $u(T_0)\in int(\Omega)$ if $\delta$ is sufficiently small. But this contradicts with the maximality of $T_0$. So $T_0\geq1$ and
 $(T,u(1))_E$ is globally flat and
$|u(1)-\bar u|=O(|l|).$

\subsubsection{Step 3}
To prove part (a) of the theorem, we only need to scale the mesh $(T,u(1)*l)_E$ to make its area equal to 1. To get the estimate in part (b), it remains to show
$$
\log |(T,u(1)*l)_E|= O(|l|).
$$
Since
\begin{align*}
&|(u(1)*l)_{ij}-\bar l_{ij}|= |(u(1)*l)_{ij}-(\bar u*l)_{ij}|+|(\bar u*l)_{ij}-\bar l_{ij}|\\
\leq&(e^{|u(1)-\bar u|}-1)(\bar u*l)_{ij}+O(l_{ij}^3)=O(|l|\cdot\bar l_{ij}).
\end{align*}
Since $(T,\bar l)_E$ is $\frac{1}{2}\epsilon$-regular, by Lemma \ref{estimate for Euclidean triangle} if
$\delta$ is sufficiently small then for any $\triangle ijk\in F$
$$
\log \frac{|(\triangle ijk, u(1)*l)_E|}{|(\triangle ijk,\bar l)_E|}=O(|l|)
$$
and
$$
\log |(T,u(1)*l)_E|
=
\log \frac{\sum_{\triangle ijk\in F}|(\triangle ijk, u(1)*l)_E|}{\sum_{\triangle ijk\in F}|(\triangle ijk,\bar l)_E|}=O(|l|).
$$

\subsection{Proof of Theorem \ref{high genus main thm}}
\label{the case of high genus}
For fix two constants $\epsilon_1,\epsilon_2>0$, we assume that $(T,l)_H$ is $(\epsilon_1,\epsilon_2)$-regular and $|l|\leq \delta$ where $\delta=\delta(M,g,\epsilon_1,\epsilon_2)<1$ is a sufficiently small constant to be determined. Simply note that $(T,l)_H$ is $\epsilon$-regular where $\epsilon=\min\{\epsilon_1,\epsilon_2\}$.
By Lemma \ref{isoperimetric},
if $\delta$ is sufficiently small there exists a geodesic triangulation $T'$ of $(M,e^{2\bar u}g)$ homotopic to $T$ relative to $V(T)=V(T')$.
Let $\bar l\in\mathbb R^{E(T)}\cong R^{E(T')}$ denote the geodesic lengths of edges of $T'$ in $(M,e^{2\bar u}g)$, and then $(T,\bar l)_H$ is isometric to $(M,e^{2\bar u}g)$ and globally hyperbolic.

For simplicity, we will frequently use the notion $a=O(b)$ to denote that if $\delta=\delta(M,g,\epsilon_1,\epsilon_2)$ is sufficiently small, then $|a|\leq C\cdot b$ for some constant $C=C(M,g,\epsilon_1,\epsilon_2)$. For example, we have that
\begin{enumerate}[label=(\alph*)]
\item $l_{ij}=O(l_{jk})$ for any $\triangle ijk\in F(T)$, and

\item $(\bar u*_hl)_{ij}=O(l_{ij})$, and

\item $\bar l_{ij}=O(l_{ij})$, and

\item $\sinh(l_{ij}/2)=O(l_{ij})$.
\end{enumerate}

The remaining of the proof is divided into two steps.
\begin{enumerate}[]
\item Firstly we show that $(T,\bar u*_hl)_H$ is very close to the globally hyperbolic triangle mesh $(T,\bar l)_H$, in the sense that
$$
(\bar u*_hl)_{ij}-\bar l_{ij}=O(l_{ij}^3)
$$
and
$$
K(\bar u)=div(x)+y
$$
for some $x\in\mathbb R^E_A$ and $y\in\mathbb R^V$ such that $x_{ij}=O(l_{ij}^2)$ and $y_i=O(l_{ij}^4)$.

\item Secondly, we construct a flow $u(t):[0,1]\rightarrow\mathbb R^V$, starting at $u(0)=\bar u$, to linearly eliminate the curvature $K(\bar u)$, i.e., to let
    $$
    K(u(t))=(1-t)K(\bar u).
    $$ Further we show that $|u'(t)|=O(|l|)$, and then $(T,u(1)*_hl)_H$ is globally hyperbolic and $u(1)-\bar u=O(|l|)$.
\end{enumerate}

The uniqueness of the discrete uniformization conformal factor is also proved by Bobenko-Pinkall-Springborn (see Theorem 6.1.6 in \cite{bobenko2015discrete}), so we omit its proof here.
\subsubsection{Part 1}
\label{hyperbolic 1}
By lemma \ref{isoperimetric}, $(T,\bar{l})_{H}$ is $\frac{1}{2}\epsilon$-regular if $\delta$ is sufficiently small. For simplicity we denote $\bar u|_{V(T)}$ as $\bar u$. By lemma \ref{cubic estimate}, we get
\begin{equation*}
\label{11}
\bar l_{ij}-(\bar u*l)_{ij}=O (l_{ij}^3).
\end{equation*}
Using the fact that $|2\sinh(\frac{x}{2})-x|\leq |x|^3$ for $|x|\leq1$, we have
\begin{equation*}
	\bar{l}_{ij}-(\bar{u}*_{h}l)_{ij}=O(l_{ij}^3).
\end{equation*}
Denote $\bar\theta^i_{jk}$ as the inner angle in $(T,\bar l)_H$, and then by Lemma \ref{estimate for hyperbolic triangle} and \ref{triangle area} and Remark \ref{remark}
$$
\alpha^i_{jk}:=\bar\theta^i_{jk}-\theta^i_{jk}(\bar u)=O(l_{ij}^2)
$$
and
$$
\alpha^i_{jk}+\alpha^j_{ik}+\alpha^k_{ij}=|(\triangle ijk,\bar u*_hl)_H|
-|(\triangle ijk,\bar l)_H|=
O(l_{ij}^2)\cdot|(\triangle ijk,\bar l)_H|=O(l_{ij}^4).
$$
So $(T,\bar u*_hl)_H$ is $\frac{1}{3}\epsilon$-regular if $\delta$ is sufficiently small.
Let $x\in\mathbb R^E_A$ and $y\in\mathbb R^V$ be such that
$$
x_{ij}=\frac{\alpha^i_{jk}-\alpha^j_{ik}}{3}+\frac{\alpha^i_{jk'}-\alpha^j_{ik'}}{3}\quad
\text{ and }
\quad
y_i=\frac{1}{3}\sum_{jk:\triangle ijk\in F(T)}(\alpha^i_{jk}+\alpha^j_{ik}+\alpha^k_{ij})
$$
where $\triangle ijk$ and $\triangle ijk'$ are adjacent triangles. Then
\begin{equation*}
div(x)_i+y_i=K_i(\bar u),
\end{equation*}
and
\begin{equation}
\label{21}
x_{ij}=O(l_{ij}^2),
\end{equation}
and
\begin{equation}
\label{22}
y_i=O(l_{ij}^4)
\end{equation}
by the fact that any vertex $i\in V(T')$ has at most $\lfloor2\pi/(\epsilon/2)\rfloor=O(1)$ neighbors.
\subsubsection{Part 2}
Let
$$
\tilde\Omega_{H}=\{u\in\bold 1^\perp:u*_{h}l\text{ satisfies the triangle inequalities and $(T,u*_{h}l)_H$ is $\frac{\epsilon}{5}$-regular}\}
$$
and
$$
\Omega_{H}=\{u\in\tilde\Omega:|u-\bar u|\leq1\text{, $(T,u*_h{l})_H$ is $\frac{\epsilon}{4}$-regular}\}.
$$

Since $(T,\bar u*_hl)_H$ is $\frac{1}{3}\epsilon$-regular, $\bar u$ is in the interior of $\Omega_{H}$.
Now consider the following ODE on $int(\tilde\Omega_{H})$,
\begin{equation}
\label{ode for high genus case}
	\left\{
	\begin{array}{lcl}
	u'(t)=(D(u)-\Delta_{\eta(u)})^{-1} K(\bar u)=(D(u)-\Delta_{\eta(u)})^{-1}(div(x)+y)\\
	u(0)=\bar u
	\end{array}
	\right.,
\end{equation}
where $D(u)$ and $\eta(u)$ are defined as in Proposition \ref{Discrete curvature equation for hyperbolic case}. For any triangle $\triangle ijk$ and $u\in\tilde\Omega_{H}$, by Lemma \ref{triangle area} and Remark \ref{remark} we have
$$
|(\triangle ijk,u*_hl)_H|=O(l_{ij}^2)
$$
and
$$
\frac{1}{2}(\pi+\theta^k_{ij}(u)-\theta^j_{ik}(u)-\theta^i_{jk}(u))=\theta^k_{ij}(u)
+\frac{1}{2}(\pi-\theta^k_{ij}(u)-\theta^j_{ik}(u)-\theta^i_{jk}(u))=\theta^k_{ij}(u)+O(l_{ij}^2).
$$
Now let $w(u)$ be defined as in Proposition \ref{Discrete curvature equation for hyperbolic case}, and then by the formula
$$
\cot A+\cot B=
\frac{\sin(A+B)}{\sin A\sin B}\geq\sin(A+B)\quad\text{for any $A,B\in(0,\pi)$},
$$
we have that if $\delta$ is sufficiently small and $u\in \tilde\Omega_{H}$
$$
w_{ij}(u)
\geq\frac{1}{2}\sin(\theta^k_{ij}+\theta^{k'}_{ij}+O(l_{ij}^2))\geq\frac{1}{2}\sin\frac{\epsilon}{5}+O(l_{ij}^2)
\geq\frac{1}{4}\sin\frac{\epsilon}{5},
$$
and
\begin{equation}
\label{24}
D_{ii}(u)\geq 2w_{ij}\tanh^2\frac{(u*_h l)_{ij}}{2}\geq\epsilon'l_{ij}^2,
\quad
\text{ and }
\quad \eta_{ij}(u)\geq\frac{1}{8}\sin\frac{\epsilon}{5}
\end{equation}
for some constant $\epsilon'=\epsilon'(M,g,\epsilon_1,\epsilon_2)>0$.

The right-hand side of equation (\ref{ode for high genus case}) is a smooth function of $u$, so the ODE (\ref{ode for high genus case}) has a unique solution $u(t)$ and
$$
K(u(t))=(1-t)K(\bar u)
$$
by Proposition \ref{Discrete curvature equation for hyperbolic case}.
Assume the maximum existing open interval of $u(t)\in\Omega_{H}$ is $[0,T_0)$ where $T_0\in(0,\infty]$.
By Lemma \ref{isoperimetric}, when $\delta$ is sufficiently small, $(T,l)$ is $C$-isoperimetric for some constant $C=C(M,g,\epsilon_1,\epsilon_2)$. Then for any $u\in\Omega_{H}$, $(T,u*_hl)$ is $(e^{4(|\bar u|+1)}C)$-isoperimetric by the fact that $|u|\leq|\bar u|+1$ and
$$
\frac{\sinh a}{a}\geq\frac{\sinh b}{b}
$$
for any $a\geq b>0$.
By Lemma \ref{triangle area} and Remark \ref{remark}, it is not difficult to see
$$
|V|_l=O(|V|_{\bar l})=O(|(T,\bar l)_H|)=O(1)\quad\text{ and }\quad
1=O(|(T,\bar l)_H|)=O(|V|_{\bar l})=O(|V|_l).
$$
Then by Lemma \ref{estimate for divergence operator} and equation (\ref{21})(\ref{22})(\ref{24}), for any $t\in[0,T_0)$
\begin{equation}
|u'(t)| =O(|l| \cdot|V|_l^{1/2})=O(|l|).
\end{equation}
By Lemma \ref{estimate for hyperbolic triangle},
if $T_0<1$,
\begin{equation}
|u(T_0)-\bar u|=O(|l|)\quad\text{ and }\quad\theta^i_{jk}(u(T_0))-\theta^i_{jk}(\bar u)=O(|l|),
\end{equation}
and then $u(T_0)\in int(\Omega_{H})$ if $\delta$ is sufficiently small. But this contradicts with the maximality of $T_0$. So $T_0\geq1$ and
 $(T,u(1))_H$ is hyperbolic and
$|u(1)-\bar u|=O(|l|).$

\section{Proof of the Geometric Lemmas}
\label{Proof of geometric lemmas}
We prove Lemma \ref{triangle area} and \ref{cubic estimate}
in Subsection 5.1 and 5.2 respectively, and introduce more lemmas in Subsection 5.3, and then prove part (a) and part (b) of Lemma \ref{isoperimetric} in Subsection 5.4 and 5.5 respectively.
\subsection{Proof of Lemma \ref{triangle area}}
Recall that
\begin{manualtheorem}{Lemma 4.1}
Suppose $\triangle_E ABC$, $\triangle_H ABC$ and $\triangle_S ABC$ are Euclidean and hyperbolic and spherical triangles respectively, with the same edge lengths $a,b,c<0.1$.

(a)
If all the inner angles in $\triangle_E ABC$ are at least $\epsilon>0$, then for any $P\in\{E,H,S\}$,
$$
\frac{\epsilon}{8} a^2\leq |\triangle_P ABC|\leq \frac{1}{\epsilon}a^2.
$$

(b)
Assume $M_a$ is the middle point of $BC$, and $M_b$ is the middle point of $AC$, and $\triangle_P CM_aM_b$ is the geodesic triangle in $\triangle_PABC$ with vertiecs $C,M_a,M_b$, where $P\in\{E,H,S\}$.
Then
$$
|\triangle_P CM_aM_b|\geq \frac{1}{5}|\triangle_P ABC|
$$
for any $P\in\{E,H,S\}$.\end{manualtheorem}
\begin{proof}[Proof of (a)]
We begin with three well known Heron's formulae for Euclidean, hyperbolic and spherical triangles.
$$
|\triangle_E ABC|^2=s(s-a)(s-b)(s-c),
$$
\begin{equation}
\label{27}
\tan^2\frac{|\triangle_H ABC|}{4}=\tanh\frac{s}{2}\tanh\frac{s-a}{2}\tanh\frac{s-b}{2}\tanh\frac{s-c}{2},
\end{equation}
\begin{equation}
\label{28}
\tan^2\frac{|\triangle_S ABC|}{4}=\tan\frac{s}{2}\tan\frac{s-a}{2}\tan\frac{s-b}{2}\tan\frac{s-c}{2},
\end{equation}
where $s=\frac{a+b+c}{2}$.\par
The hyperbolic Heron's formula can be found in Theorem 1.1 in \cite{mednykh2012brahmagupta}, and the spherical one is also called L'Huilier's Theorem and can be found in Section 4.19.2 in \cite{zwillinger2002crc}.

Notice that $|\triangle_EABC|\leq a^2+b^2+c^2\leq0.03$, and for $x\in[0,0.1]$,
\begin{equation*}
\frac{\tanh x}{x}\in(0.99,1)\quad\text{ and }\quad\frac{\tan x}{x}\in(1,1.01).
\end{equation*}
So by the three parallel Heron's formulae and simple approximation estimates, we only need to show the following stronger estimates (\ref{19}) and (\ref{20}) for the Euclidean case.
By the law of sines in the Euclidean triangle $\triangle_E ABC$,
$$
b=\frac{a\sin\angle_E B}{\sin\angle_E A}\leq \frac{a}{\sin\epsilon}\leq \frac{\pi}{2\epsilon}a.
$$
So
\begin{equation}
\label{19}
|\triangle_E ABC|=\frac{1}{2}ab\sin C\leq\frac{1}{2}a\cdot \frac{\pi }{2\epsilon}a=\frac{\pi}{4}\frac{a^2}{\epsilon}.
\end{equation}
By the triangle inequality, we may assume $b\geq a/2$ without loss of generality, and then
\begin{equation}
\label{20}
|\triangle_E ABC|=\frac{1}{2}ab\sin C\geq\frac{1}{2}a\cdot\frac{a}{2}\cdot\sin\epsilon\geq\frac{\epsilon}{2\pi}a^2.
\end{equation}

\end{proof}
\begin{proof}[Proof of (b)]
The Euclidean case is obvious. To prove the hyperbolic and spherical cases, we use the following two formulae \begin{equation}
\label{1621}
\cot\frac{|\triangle_H ABC|}{2}=\frac{\coth\frac{a}{2}\coth\frac{b}{2}-\cos\angle_HC}{\sin\angle_H C},
\end{equation}
\begin{equation}
\label{1622}
\cot\frac{|\triangle_S ABC|}{2}=\frac{\cot\frac{a}{2}\cot\frac{b}{2}+\cos\angle_S C}{\sin\angle_S C},
\end{equation}
where equation (\ref{1621}) was developed in Theorem 6 of \cite{frenkel2018area}. The equation (\ref{1622}) can be obtained by
$$
\cot\frac{|\triangle_S ABC|}{2}=\cot\frac{\angle_SA+\angle_SB+\angle_SC-\pi}{2}=-\tan(\frac{\angle_S A+\angle_S B}{2}+\frac{\angle_S C}{2})
$$
and the well-known Napier's analogies
$$
\tan\frac{\angle_SA+\angle_SB}{2}=\cot\frac{C}{2}\cdot\frac{\cos\frac{a-b}{2}}{\cos\frac{a+b}{2}}.
$$
Here we only prove the hyperbolic case using equation (\ref{21}) and the proof for the spherical case is very similar.
Firstly we apply the formula (\ref{1621}) to $\triangle_HCM_aM_b$ and get
$$
\cot\frac{|\triangle_H CM_aM_b|}{2}=\frac{\coth\frac{a}{4}\coth\frac{b}{4}-\cos\angle_HC}{\sin\angle_H C}.
$$
Then
$$
\frac{\tan\frac{\triangle_HCM_aM_b}{2}}{\tan\frac{\triangle_{H}ABC}{2}}=
\frac{\coth\frac{a}{2}\coth\frac{b}{2}-\cos\angle_H C}{\coth\frac{a}{4}\coth\frac{b}{4}-\cos\angle_HC}
\geq\frac{(2/a)(2/b)-1}{(4/a)(4/b)/0.99^2+1}=\frac{4-ab}{16/0.99^2+ab}\geq\frac{1}{5}.
$$
Since $|\triangle_HC M_aM_b|\leq|\triangle_HABC|$ and $\frac{\tan x}{x}$ is increasing on $(0,\infty)$,
$$
\frac{|\triangle_H CM_aM_b|}{|\triangle_H ABC|}\geq\frac{\tan\frac{\triangle_HCM_aM_b}{2}}{\tan\frac{\triangle_{H}ABC}{2}}\geq\frac{1}{5}.
$$
\end{proof}
\subsection{Proof of Lemma \ref{cubic estimate}}

Recall that

\begin{manualtheorem}{Lemma 4.2}
Suppose $(M,g)$ is a  closed Riemannian surface, and $ u\in C^{\infty}(M)$ is a conformal factor. Then there exists $C=C(M,g,u)>0$ such that for any $x,y\in M$,
$$
|d_{e^{2 u}g}(x,y)-e^{\frac{1}{2}(u(x)+ u(y))}d_g(x,y)|\leq Cd_g(x,y)^3.
$$
\end{manualtheorem}

It suffices to prove one direction of the inequality, i.e., the following Lemma \ref{proof of cubic estimate in one direction}. Once we have Lemma \ref{proof of cubic estimate in one direction}, let $C_1=C(M,g,u)$ and $C_2=C(M,e^{2u}g,-u)$ such that for any $x,y\in M$,
$$
d_{e^{2 u}g}(x,y)\leq e^{\frac{1}{2}( u(x)+ u(y))}d_g(x,y)+ C_1d_g(x,y)^3.
$$
$$
d_{g}(x,y)\leq e^{\frac{1}{2}(-u(x)-u(y))}d_{e^{2u}g}(x,y)+ C_2d_{e^{2u}g}(x,y)^3.
$$
Then
\begin{align*}
&|d_{e^{2 u}g}(x,y)-e^{\frac{1}{2}( u(x)+ u(y))}d_g(x,y)|\\
\leq &C_1d_g(x,y)^3+C_2e^{\|u\|_\infty}d_{e^{2u}g}(x,y)^3\\
\leq&(C_1+C_2e^{4\|u\|_\infty})d_g(x,y)^3.
\end{align*}
\begin{lemma}
	\label{proof of cubic estimate in one direction}
Suppose $(M,g)$ is a  closed Riemannian surface, and $ u\in C^{\infty}(M)$ is a conformal factor. Then there exists $C=C(M,g, u)>0$ such that for any $x,y\in M$,
$$
d_{e^{2 u}g}(x,y)\leq e^{\frac{1}{2}( u(x)+ u(y))}d_g(x,y)+ Cd_g(x,y)^3.
$$

\end{lemma}
\begin{proof}
We will use the following two estimates. Assume $l>0$ and $f\in C^2[0,l]$, then
$$
\left|\frac{1}{2}[f(0)+f(l)]-f(\frac{l}{2})\right|\leq\frac{l^2}{4}\max_{0\leq t\leq l}|f''(t)|,
$$
and
$$
\left|\int_0^l f(x)dx-l\cdot f(\frac{l}{2})\right|\leq\frac{l^3}{24}\max_{0\leq t\leq l}|f''(t)|.
$$
These two estimates can be proved easily by Taylor's expansions at point $x_0=l/2$, and
the second estimate is the so-called mid-point rule approximation of definite integrals.
Now let $l=d_g(x,y)$ and $\gamma:[0,l]\rightarrow(M,g)$ be a shortest geodesic connecting $x,y$, and it suffices to prove
$$
d_{e^{2 u}g}(x,y)-l\cdot e^{\frac{1}{2}(u(x)+u(y))}
\leq Cl^3
$$
for some constant $C=C(M,g,u)$.
Let $h(t)=u(\gamma(t))$ and then
by the two estimates above,
\begin{align*}
&d_{e^{2 u}g}(x,y)-l\cdot e^{\frac{1}{2}(u(x)+u(y))}\\
\leq& s_{e^{2u}g}(\gamma)-l\cdot e^{\frac{1}{2}(u(x)+u(y))}\\
=&\left[
\int_0^l e^{h(t)}dt-l\cdot e^{h(l/2)}
\right]
+
\left[
l\cdot e^{h(l/2)}-l\cdot e^{\frac{1}{2}(h(0)+h(l))}
\right]\\
\leq
&\frac{l^3}{24}\max_{0\leq t\leq l}\left|\left(e^{h(t)}\right)''\right|
+l\cdot e^{\xi}\cdot\left|h(\frac{l}{2})-\frac{1}{2}[h(0)+h(l)]   \right|
\end{align*}
where $\xi$ is between $h(l/2)$ and $\frac{1}{2}[h(0)+h(l)]$, and
$$
\left|h(\frac{l}{2})-\frac{1}{2}[h(0)+h(l)]\right|\leq\frac{l^2}{4}\max_{0\leq t\leq l}|h''(t)|.
$$

By the compactness of $M$, and the fact that $h(t),h'(t),h''(t)$ can be expressed in terms of $\{u,\nabla u,Hess(u),\gamma,\gamma',\gamma''\}$ under local coordinates $(v^1,v^2)$, we only need to show that on a small domain $U$ whose closure is a compact subset of a coordinate domain, $\|u\|_\infty, \|\nabla u\|_\infty$, $\|Hess(u)\|_\infty$, $\|\gamma\|_\infty,\|\gamma'\|_\infty,\|\gamma''\|_\infty$ are all bounded by a constant $C(M,g,u,U)$. It is obvious that $\|u\|_\infty,\|\nabla u\|_\infty,\|Hess(u)\|_\infty$, $\|\gamma\|_\infty$ are bounded by a constant, by the compactness of $\bar U$. $\|\gamma'\|_\infty$ is bounded by a constant since $\langle\gamma'(t),\gamma'(t)\rangle_g=1$, and then $\|\gamma''(t)\|_\infty$ is also bounded by a constant,
by the geodesic equation
$$
\frac{d^2}{dt^2}\gamma^i+\sum_{j,k}\Gamma^i_{jk}(\frac{d}{dt}\gamma^j)(\frac{d}{dt}\gamma^k)=0.
$$
\end{proof}
\subsection{Lemmas for the Proof of Lemma \ref{isoperimetric}}

\begin{lemma}
\label{Toponogov}
Assume $\triangle ABC,\triangle A'B'C',\triangle A''B''C''$ are geodesic Riemannian triangles with the same edge lengths $a,b,c$, and $\triangle A'B'C'$ has constant curvature $-K<0$, and $\triangle A''B''C''$ has constant curvature $K>0$, and the curvature of $\triangle ABC$ is always in $[-K,K]$, and
$$
\max\{a,b,c\}<\frac{\pi}{2\sqrt K}.
$$
Then we have
$$
A'\leq A\leq A''.
$$
\end{lemma}
This is a combination of the well-known Toponogov comparison theorem and the CAT(K) Theorem.
See Theorem 79 on page 339 in \cite{petersen2006riemannian} for the Toponogov comparison theorem,
and Characterization Theorem on page 704 in \cite{alexander1993geometric} or Theorem 1A.6 on page 173 in \cite{bridson1999metric} for the CAT(K) Theorem.
We omit the proof here.
\begin{lemma}
\label{5.3}
Assume $\triangle ABC$ and $\triangle A'B'C'$ are two geodesic Riemannian triangles with the same edge lengths $a,b,c$, and the Gaussian curvature on $\triangle ABC$ and $\triangle A'B'C'$ are both bounded in $(-K,K)$, and $\max\{a,b,c\}<\frac{\pi}{3\sqrt K}$.
Then
$$
|A'-A|\leq  2(a+b+c)^2K.
$$
\end{lemma}
\begin{proof}
By Lemma \ref{Toponogov}, without loss of generality, we may assume that $\triangle ABC$ has constant curvature $-K$ and $\triangle A'B'C'$ has constant curvature $K$. Then
$$
A'-A>0,\quad B'-B>0,\quad C'-C>0,
$$
and by the Gauss-Bonnet theorem
$$
0<A'-A\leq(A'+B'+C')-(A+B+C)=K\cdot\big(|\triangle A'B'C'|+|\triangle ABC|\big).
$$
By a simple scaling, the Heron's formulae (\ref{27}) and (\ref{28}) can be generalized to the following
\begin{equation*}
\tan^2\frac{|\triangle ABC|\cdot K}{4}=\tanh\frac{s\sqrt K}{2}\tanh\frac{(s-a)\sqrt K}{2}\tanh\frac{(s-b)\sqrt K}{2}\tanh\frac{(s-c)\sqrt K}{2},
\end{equation*}
\begin{equation*}
\tan^2\frac{|\triangle A'B'C'|\cdot K}{4}=\tan\frac{s\sqrt K}{2}\tan\frac{(s-a)\sqrt K}{2}\tan\frac{(s-b)\sqrt K}{2}\tan\frac{(s-c)\sqrt K}{2}\leq s^4K^2,
\end{equation*}
where $s=(a+b+c)/2$.
So
$$
|\triangle ABC|\leq|\triangle A'B'C'|\leq\frac{4}{K}\tan\frac{|\triangle A'B'C'|\cdot K}{4}\leq\frac{4}{K}\cdot s^2K=(a+b+c)^2
$$
and we are done.
\end{proof}

\begin{lemma}
\label{lemma 5.2}
Suppose $(M,g)$ is a  closed Riemannian surface, and $ u\in C^{\infty}(M)$ is a conformal factor, then for any $\epsilon>0$, there exists $\delta=\delta(M,g,u)>0$ such that for any $x,y\in M$ with $d_g(x,y)<\delta$,
\begin{enumerate}[label=(\alph*)]
\item there exists a unique shortest geodesic segment $l$ in $(M,g)$, and $l'$ in $(M,e^{2u}g)$, connecting $x$ and $y$, and
\item the angle between $l$ and $l'$ at $x$, measured in $(M,g)$, is less or equal to $\epsilon$.
\end{enumerate}
\end{lemma}
\begin{proof}
Assume $K(x)$ is the Gaussian curvature of $(M,g)$ at $x$, and $\|K\|_\infty=\max_{x\in M}|K(x)|$.
It is easy to find a sufficiently small $\delta$ such that (a) is satisfied, and for any $x\in M$,
$$
|B_g(x,\delta)|_g\cdot\|K\|_\infty<\epsilon/2.
$$
Consider the unit circle bundle
$$
A=\{(x,\vec a)\in TM:x\in M,\|\vec a\|_{e^{2u}g}=1\},
$$
and assume we are in local coordinates $(v_1,v_2)$, and $\Gamma^i_{jk}$ are Christoffel symbols for $g$, and $\tilde\Gamma^i_{jk}$ are Christoffel symbols for $e^{2u}g$.
Then for any geodesic $\gamma(t)=(v_1(t),v_2(t))$ in $(M,e^{2u}g)$ with $\gamma(0)=x$ and $\gamma'(0)= \vec a$, the geodesic curvature $k_g$ of $\gamma$ in $(M,g)$ at point $x$ is
$$
\frac{
-\sqrt{g_{11}g_{22}-g_{12}^2}(-\Gamma^2_{11} \dot v_1^3+\Gamma^1_{22}\dot v_2^3-
(2\Gamma^2_{11}-\Gamma^1_{11})\dot v_1^2\dot v_2+(2\Gamma^1_{12}-\Gamma^2_{22})\dot v_1\dot v_2^2+\ddot v_1\dot v_2-\ddot v_2\dot v_1)}
{\|\vec a\|_g^3}
$$
(see Theorem 17.19 in \cite{gray2006modern} for a proof).
Here $(\dot v_1,\dot v_2)=\vec a$, and $\ddot v_1,\ddot v_2$ are determined by $(\dot v_1,\dot v_2)$ through the geodesic equations
$$
\ddot v_i+\sum_{j,k}\tilde\Gamma^i_{jk}\dot v_j\dot v_k=0.
$$
By this way $k_g$ can be viewed as a smooth function of $(x,\vec a)$ defined on the compact manifold $A$, and thus is bounded by $[-C,C]$ for some constant $C=C(M,g,u)$.

\begin{figure}
	\label{geodesic intersection}
	\centering
	\includegraphics[width=0.6\textwidth]{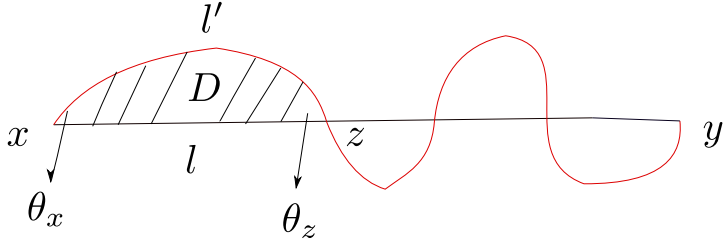}
	\caption{Geodesics in the proof of Lemma \ref{lemma 5.2}}
	\end{figure}
As shown in Figure 2, assume $l$ and $l'$ start at $x$ and first meet at a point $z$. Let $l_0$ (\emph{resp.} $l_0'$) be the part of $l$ (\emph{resp.} $l'$) between $x$ and $z$, and then by the Jordan-Schoenflies theorem $l_0\cup l_0'$ bounds a small closed disk $D$. Let $\theta_x$ (\emph{resp.} $\theta_z$) be the intersecting angle of $l_0$ and $l_0'$ at $x$ (\emph{resp.} at $z$).
Then by the Gauss-Bonnet theorem
$$
\int_DK dA_g+\int_{l_0}k_gds_g+\int_{l'_0}k_g ds_g+(\pi-\theta_x)+(\pi-\theta_z)=2\pi
$$
and
\begin{equation*}
\theta_x\leq\theta_x+\theta_z=\int_D KdA_g+\int_{l_0'}k_g ds_g\leq \|K\|_\infty\cdot |D|_g+C\cdot s_g(l')
\leq\frac{\epsilon}{2}+C\cdot s_g(l')
\end{equation*}
where
$$
s_g(l')\leq e^{\|u\|_\infty}\cdot s_{e^{2u}g}(l')\leq e^{\|u\|_\infty}\cdot s_{e^{2u}g}(l)
\leq e^{2\|u\|_\infty}\cdot s_g(l)\leq e^{2\|u\|_\infty}\cdot\delta.
$$
So $\theta_x\leq\epsilon$ if $\delta\leq\epsilon /(2Ce^{2\|u\|_\infty})$.
\end{proof}
\subsection{Proof of Part (a) of Lemma \ref{isoperimetric}}

Recall that
\begin{manualtheorem}{Part (a) of Lemma \ref{isoperimetric}}
Suppose $(M,g)$ is a closed Riemannian surface,
and $T$ is a geodesic triangulation of $(M,g)$, and $l\in\mathbb R_{>0}^{E(T)}$ denotes the geodesic lengths of the edges,
and $(T,l)_E$ or $(T,l)_H$ is $\epsilon$-regular.
\begin{enumerate}[label=(\alph*)]

\item Given a conformal factor $ u\in C^{\infty}(M)$, there exists a constant $\delta=\delta(M,g,u,\epsilon)>0$ such that if $|l|<\delta$ then
there exists a geodesic triangulation $T'$ in $(M,e^{2u}g)$ such that $V(T')=V(T)$, and $T'$ is homotopic to $T$ relative to $V(T)$. Further $(T,\bar l)_E$ and $(T,\bar l)_H$ are $\frac{1}{2}\epsilon$-regular where
$\bar l\in\mathbb R^{E(T)}\cong\mathbb R^{E(T')}$ denotes the geodesic lengths of the edges of $T'$ in $(M,e^{2u}g)$.
\end{enumerate}
\end{manualtheorem}

\begin{proof}[Proof of Part (a) of Lemma \ref{isoperimetric}]
Denote
\begin{enumerate}
\item $\theta^i_{jk}(M)$ as the inner angle of the geodesic triangle in $F(T)$ in $(M,g)$, and

\item $\theta^i_{jk}(E)$ as the inner angle in $(T,l)_E$, and

\item ${\theta}^i_{jk}(H)$ as the inner angle in $(T,l)_H$, and

\item ${\bar\theta}^i_{jk}(M)$ as the inner angle of the geodesic triangle in $F(T')$ in $(M,e^{2u}g)$, and

\item $\bar\theta^i_{jk}(E)$ as the inner angle in $(T,\bar l)_E$, and

\item ${\bar\theta}^i_{jk}(H)$ as the inner angle in $(T,\bar l)_H$.
\end{enumerate}
By Lemma \ref{5.3} and \ref{lemma 5.2}, if $\delta(M,g,u,\epsilon)$ is sufficiently small, then
\begin{enumerate}[label=(\alph*)]
\item $|\theta^i_{jk}(M)-{\theta}^i_{jk}(E)|\leq \epsilon/12$ and $|\theta^i_{jk}(M)-{\theta}^i_{jk}(H)|\leq \epsilon/12$, and

\item for any $ij\in E(T)$ there exists a unique shortest geodesic $e_{ij}$ in $(M,e^{2u}g)$ connecting $i,j$, and

\item for each $\triangle ijk\in F(T)$, $e_{ij},e_{ik},e_{jk}$ bounds a geodesic triangle $F_{ijk}$ in $(M,e^{2u}g)$, and

\item $|{\bar\theta}^i_{jk}(M)-{\theta}^i_{jk}(M)|\leq \epsilon/12$,
$|{\bar\theta}^i_{jk}(E)-\bar{\theta}^i_{jk}(M)|\leq \epsilon/12$,
$|{\bar\theta}^i_{jk}(H)-\bar{\theta}^i_{jk}(M)|\leq \epsilon/12$, and thus
$(T,\bar l)_E$ and $(T,\bar l)_H$ are $\epsilon/2$-regular, and

\item for any vertex $i\in V(T)=V(T')$, its adjacent edges $\{ij\}_{j\sim i}$ in $T$ are placed in the same order as the adjacent edges $\{e_{ij}\}_{j\sim i}$ in $T'$, i.e., there are no folding triangles in $T'$.
\end{enumerate}
We can define a continuous map $f:M\rightarrow M$ such that
\begin{enumerate}
\item $f(i)=i$ for any $i\in V$, and

\item for any edge ${ij}\in E(T)$, $f$ is a homeomorphism from ${ij}$ to $e_{ij}$, and

\item for any $\triangle ijk\in F(T)$, $f$ is a homeomorphism from $\triangle ijk$ to $F_{ijk}$.
\end{enumerate}
Then by above property (e), $f$ is locally a homeomorphism. Further if $\delta$ is sufficiently small, $f$ is homotopic to the identity. Therefore $f$ is a global homeomorphism and $T'=(V,\{e_{ij}\},\{F_{ijk}\})$ is a triangulation of $M$.
\end{proof}

\subsection{Proof of Part (b) of Lemma \ref{isoperimetric}}Recall that
\begin{manualtheorem}{Part (b) of Lemma \ref{isoperimetric}}
Suppose $(M,g)$ is a closed Riemannian surface,
and $T$ is a geodesic triangulation of $(M,g)$, and $l\in\mathbb R^{E(T)}$ denotes the geodesic lengths of the edges, and $(T,l)_E$ or $(T,l)_H$ is $\epsilon$-regular.

\quad (b) There exists a constant $\delta=\delta(M,g,\epsilon)$ such that if $|l|<\delta$,
$(T,l)$ is $C$-isoperimetric for some constant $C=C(M,g,\epsilon)$.
\end{manualtheorem}
We first prove a continuous version.
\begin{lemma}\label{lemma 5.3}
Suppose $(M,g)$ is a closed Riemannian surface, and $\Omega\subset M$ is an open domain with $\partial\Omega$ being a finite disjoint union of piecewise smooth Jordan curves, then there exists a constant $C=C(M,g)>0$ such that
$$
\min\{|\Omega|,|M-\Omega|\}\leq C L^2
$$
where $L=s(\partial\Omega)$ denotes the length of $\partial\Omega$ in $(M,g)$.
\end{lemma}
\begin{proof}
If $\Omega$ is simply connected, then it is well known (See Theorem 4.3 in \cite{osserman1978isoperimetric}) that
$$
L^2\geq |\Omega|(4\pi-2\int_\Omega K^+)
$$
where $K^+(p)=\max\{0,K(p)\}$.
Pick $r=r(M,g)>0$ smaller than the injectivity radius of $(M,g)$, such that
$$
|{B(p,r)}|\cdot\|K\|_\infty\leq \pi
$$
for any $p\in M$.\par
Now we pick our constant
$$
C(M,g)=\max\{\frac{2}{\pi},\frac{|M|}{r^2}\}.
$$
If $L\geq r$, then $CL^2\geq |M|$ and we are done. If $\Omega\subset B(p,r)$ for some $p\in M$ and is connected, then without loss of generality we may assume $\Omega$ is simply connected by filling up the holes, and then
\begin{equation}
\label{33}
CL^2\geq \frac{2}{\pi}\cdot |\Omega|(4\pi-2\int_\Omega K^+)\geq |\Omega|\big(8-\frac{4}{\pi}\cdot |B(p,r)|\cdot\|K\|_\infty\big)\geq4|\Omega|
\end{equation}
and we are done.

If $\Omega$ has multiple connected components $\Omega_1,...,\Omega_n$ with the boundary lengths $L_1,...,L_n$ respectively, such that each $\Omega_i$ is in some Riemannian disk $B(p,r)$, then
$L\geq(L_1+...+L_n)/2$ since any component of $\partial\Omega$ is on at most two $\partial\Omega_i$'s. So by equation (\ref{33})

\begin{equation}
\label{34}
CL^2\geq \frac{1}{4}\sum_{i=1}^n CL_i^2\geq\sum_{i=1}^n |\Omega_i|=|\Omega|
\end{equation}
and we are done.

Now we assume $L<r$ and $\partial\Omega$ contains Jordan curves $\gamma_1,...,\gamma_n$ with lengths $L_1,...,L_n$ respectively. Since $L_i\leq r$, $\gamma_i$ is in some Riemannian disk $B(p,r)$.
By the Jordan-Schoenflies theorem, $\gamma_i$ separates $M$ into a smaller domain $U_i\subset B(p,r)$ and a larger domain $V_i=M-\bar U_i$, and $\bar U_i$ is a topological closed disk. For any $i\neq j$, since $\gamma_i$ and $\gamma_j$ are disjoint, $\bar U_i\subset \bar U_j$ or $\bar U_j\subset \bar U_i$ or $\bar U_i\cap\bar U_j=\emptyset$. So $\cup_{i=1}^n\bar U_i$ is a finite disjoint union of topological disks, and thus
 $M-\cup_{i=1}^n  \bar U_i$ is connected.
If $\Omega\subset\cup_{i=1}^n U_i$, then by equation (\ref{34}) we are done.
Otherwise, $M-\cup_{i=1}^n \bar U_i\subset\Omega$ and $M-\bar\Omega\subset\cup_{i=1}^n U_i$, and again
by equation (\ref{34}) $CL^2\geq |M-\bar\Omega|$ and we are done.
\end{proof}
\begin{figure}
  \centering
  \label{isoperimetric figure}
  \includegraphics[width=4.5in]{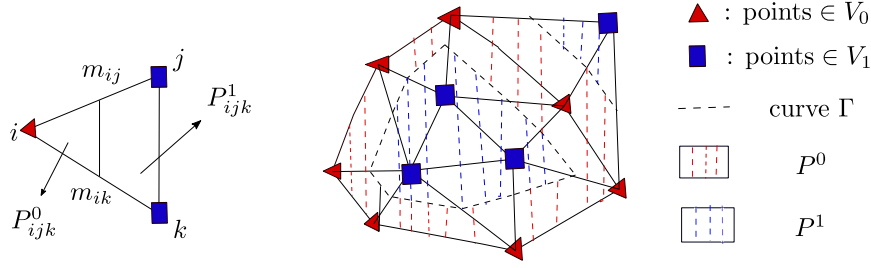}
  \caption{Decomposition of a triangulated surface}
\end{figure}
Now we prove part (b) of Lemma \ref{isoperimetric}  for the special cases that $(M,g)$ has constant curvature $0$ or $\pm 1$.
\begin{proof}[Proof of Part (b) of Lemma \ref{isoperimetric} for the cases of constant curvature $0$ or $\pm1$]
In this proof each triangle $\triangle ijk\in F(T)$ is identified as a geodesic triangle in $(M,g)$.
Assume $\delta<0.1$, $V_0\subset V$, and $V_1=V-V_0$. Let
$$
E_0=\{ij\in E:i,j\in V_0\},\quad
E_1=\{ij\in E:i,j\in V_1\}.
$$
Notice that $\partial V_{0}=\partial V_1$
and $E=E_0\cup E_1\cup\partial V_0$ is a disjoint union.

For any triangle $\triangle ijk\in F(T)$, 0 or 2 of its edges are in $\partial V_0$. So $F(T)=F_0\cup F_2$ where
\begin{align*}
F_0=&\{\triangle{ijk}\in F(T):\triangle ijk\text{ has 0 edges in }\partial V_0\},\text{ and}\\
F_2=&\{\triangle{ijk}\in F(T):\triangle ijk\text{ has 2 edges in }\partial V_0\}.
\end{align*}
If $\triangle ijk\in F_2$ and $ij,ik\in\partial V_0$, let $\gamma_{ijk}$ be the geodesic segment
in $\triangle ijk$ connecting the middle points $m_{ij}$ of $ij$, and $m_{ik}$ of $ik$. Then
by the triangle inequality $\frac{1}{2}(l_{ij}+l_{ik})\geq s(\gamma_{ijk})$.
$\gamma_{ijk}$ cut $\triangle ijk$ into two relative open domains $P_{ijk}^0$ and $P_{ijk}^1$ such that $P_{ijk}^0\cap V_0\neq\emptyset$ and $P_{ijk}^1\cap V_1\neq\emptyset$. Given $\triangle ijk\in F_0$,
\begin{enumerate}
\item if $i,j,k\in V_0$, denote $P^0_{ijk}=\triangle ijk$ and $P^1_{ijk}=\emptyset$, and

\item if $i,j,k\in V_1$, denote $P^1_{ijk}=\triangle ijk$ and $P^0_{ijk}=\emptyset$.
\end{enumerate}
The union
$$
\Gamma=\bigcup_{\triangle ijk\in F_2}\gamma_{ijk}
$$
is a finite disjoint union  of piecewise smooth Jordan curves in $(M,g)$, and
$$
P^0=\bigcup_{\triangle ijk\in F(T)}P^0_{ijk},\quad\text{ and }\quad
P^1=\bigcup_{\triangle ijk\in F(T)}P^1_{ijk}
$$
are two open domains of $(M,g)$ such that $\partial P^0=\Gamma$ and $P^1=M-\bar P^0$.
The above notations are shown in Figure 3.
By Lemma \ref{lemma 5.3}, it suffices to prove that if $\delta<0.1$,
\begin{equation}
\label{M removing bar U}
|P^1|\geq
\frac{\epsilon}{60}(|V|_l-|V_0|_l),
\end{equation}
and
\begin{equation}
\label{u geq l perimeter v0}
|P^0|\geq \frac{\epsilon}{60}|V_0|_l,
\end{equation}
and
\begin{equation}
\label{lg gamma leq l perimeter boundary v_0}
s(\Gamma)\leq|\partial V_0|_l.
\end{equation}

By part (b) of Lemma \ref{triangle area} and Remark \ref{remark}, we have that
\begin{align*}
|V|_l-|V_0|_l&=\sum_{ij\in E_1\cup\partial V_0}l_{ij}^2\leq \frac{4}{\epsilon}\sum_{ij\in E_1\cup\partial V_0}(|\triangle ijk|+|\triangle ijk'|)
\leq\frac{12}{\epsilon}\sum_{\triangle ijk\in F:\triangle ijk\cap P^1\neq\emptyset}|\triangle ijk|\\
\leq &\frac{60}{\epsilon}\sum_{\triangle ijk\in F}
|P^1_{ijk}|=\frac{60}{\epsilon}|P^1|,
\end{align*}
and
$$
|V_0|_l=\sum_{ij\in E_0}l_{ij}^2\leq
\sum_{ij\in E_0\cup\partial V_0}l_{ij}^2\leq\frac{60}{\epsilon}|P^0|,
$$
and
$$
|\partial V_0|_l=\sum_{ij\in\partial V_0}l_{ij}
=\sum_{\triangle ijk\in F_2:jk\notin\partial V_0}\frac{1}{2}(l_{ij}+l_{ik})
\geq\sum_{\triangle ijk\in F_2}s(\gamma_{ijk})=s(\Gamma).
$$
\end{proof}
Now let us prove part (b) of Lemma \ref{isoperimetric} for general smooth surfaces.
\begin{proof}[Proof of Part (b) of Lemma \ref{isoperimetric}]
By the Uniformization theorem, there exists $u=u_{M,g}\in C^\infty(M)$ such that $e^{2u}g$ has constant curvature $\pm1$ or 0.
By part (a) of Lemma \ref{isoperimetric}, if $\delta$ is sufficiently small, we can find a geodesic triangulation $T'$ in $(M,e^{2u}g)$ such that $V(T)=V(T')$, and $T,T'$ are homotopic relative to $V$. Further by the inequalities in (a) and (d) in the proof of Lemma \ref{isoperimetric} (a), if $\delta$ is sufficiently small,
any inner angle of $T'$ in $(M,e^{2u}g)$ is at least $\epsilon/2$. Let $\bar l\in\mathbb R^{E(T)}\cong\mathbb R^{E(T')}$ denote the geodesic edge lengths in $(M,e^{2u}g)$. Then by our result on surfaces of constant curvature $\pm1$ or $0$, if $\delta=\delta(M,e^{2u}g)$ is sufficiently small, $(T,\bar l)$ is $C$-isoperimetric for some constant $C=C(M,e^{2u}g)>0$.
Since
$
e^{-\|u\|_\infty}\leq{\bar l_{ij}}/{l_{ij}}\leq e^{\|u\|_\infty},
$
$(T,l)$ is $(e^{4\|u\|_\infty}C)$-isoperimetric.
\end{proof}

\section{Proof of the Discrete Elliptic Estimate}
\label{proof of the key estimates on graphs}
Recall that
\begin{manualtheorem}{Lemma 2.3}
Assume $(G,l)$ is $C_1$-isoperimetric, and $x\in\mathbb R^E_A,\eta\in\mathbb R^E_{>0},C_2>0,C_3>0$ are such that
\begin{enumerate}[label=(\roman*)]
\item $|x_{ij}|\leq C_2 l_{ij}^2$ for any $ij\in E$, and

\item $\eta_{ij}\geq C_3$ for any  $ij\in E$.
\end{enumerate}
Then
\begin{equation*}
|\Delta^{-1}_\eta\circ div (x))|\leq \frac{4C_2\sqrt{C_1+1}}{C_3}|l|\cdot|V|_l^{1/2}.
\end{equation*}
Further if $y\in\mathbb R^V$ and $C_4>0$ and $D\in\mathbb R^{V\times V}$ is a diagonal matrix such that
$$
|y_i|< C_4 D_{ii}|l|\cdot|V|_l^{1/2}
$$
for any $i\in V$, then
\begin{equation*}
|(D-\Delta_\eta)^{-1} (div(x)+y)|\leq
\left(C_4+\frac{8C_2\sqrt{C_1+1}}{C_3}\right)|l|\cdot|V|_l^{1/2}.
\end{equation*}
\end{manualtheorem}
We will first prove Lemma \ref{estimate for divergence operator} assuming  Lemma \ref{Proof for square graph estimate}, and then prove Lemma \ref{Proof for square graph estimate}.

\begin{proof}
Assume
\begin{enumerate}
\item $z=\Delta^{-1}(div (x))$, and

\item $a,b\in V$ are such that $z_a=\max_{i}z_i\geq0$ and $z_b=\min_iz_i\leq0$, and $a\neq b$, and

\item $u\in\mathbb R^V$ is such that
$$
(\Delta u)_{a}=1,\quad\text{and}\quad (\Delta u)_{b}=-1,\quad\text{and}\quad (\Delta u)_i=0 \quad\forall i\neq a,b.
$$
\end{enumerate}
By the Green's identity Lemma \ref{green's identity} and Lemma \ref{Proof for square graph estimate},

\begin{align*}
&|z|\leq z_a-z_b=\sum_{i}z_i(\Delta u)_i=\sum_{i}u_i(\Delta z)_i=\sum_{i}u_i\cdot div(x)_i
\\
=&\sum_i u_i\sum_{j:j\sim i}x_{ij}=
\sum_{ij\in E}(u_i-u_j)\cdot x_{ij}
\leq C_2\sum_{ij\in E}|u_i-u_j|\cdot l_{ij}^2
\leq \frac{4C_2\sqrt{C_1+1}}{C_3}|l|\cdot|V|_l^{1/2}.
\end{align*}

Let
$$
w=(D-\Delta)^{-1}(div(x)+y)+z,$$
and then
\begin{equation}
\label{40}
(D-\Delta)w=div(x)+y+(D-\Delta)z=y+Dz.
\end{equation}

Assume
$w_i=\max_{k}w_k$, and then by comparing the $i$-th component in (\ref{40}) we have
$$
D_{ii} w_i\leq((D-\Delta)w)_i= y_i+D_{ii}z_i
\leq y_i+D_{ii}|z|.
$$
So
$$
\max_k w_k=w_i\leq |z|+y_i/D_{ii}\leq |z|+\max_k(y_k/D_{kk})
$$
and
similarly we also have that
$$
\min_k w_k\geq -|z|+\min_k(y_k/D_{kk}).
$$
So
$$
|(D-\Delta)^{-1}(div(x)+y)|\leq|w|+|z|\leq2|z|+\max_k(|y_k|/D_{kk})
$$ and we are done.

\end{proof}

	\begin{lemma}
		\label{Proof for square graph estimate}
Assume $(G,l)$ is $C_1$-isoperimetric, and the weight $\eta\in\mathbb R^E_{>0}$ satisfies that $\eta_{ij}\geq C_2$ for some constant $C_2>0$, and $u\in\mathbb R^V$ satisfies that
$$
(\Delta u)_{a}=1,\quad\text{and}\quad (\Delta u)_{b}=-1,\quad\text{and}\quad (\Delta u)_i=0 \quad\forall i\neq a,b.
$$
then
	\begin{equation*}
	\sum_{ij\in E}l_{ij}^{2}|u_i-u_j|\leq \frac{4\sqrt{C_1+1}}{C_2}|l|\cdot |V|_{l}^{1/2}.
	\end{equation*}
\end{lemma}
\begin{proof}
We consider the 1-skeleton $X$ of the graph $G$ with edge length $l$. More specifically $X$ can be constructed as follows. Let $\tilde X$ be a disjoint union of line segments $\{e_{ij}:ij\in E\}$ where each $e_{ij}$ has length $l_{ij}$ and two endpoints $v_{ij}^i,v_{ij}^j$. Then we obtain a connected quotient space $X$ by
identifying the points in $v_i:=\{v^i_{ij}:ij\in E\}$ for any $i\in V$.

Assume $\mu$ is the standard 1-dimensional Lebesgue measure on $X$ such that $\mu(e_{ij})=l_{ij}$. Let $\nu$ be another measure on $X$ such that ${d\nu}/{d\mu}\equiv l_{ij}$ on edge $e_{ij}$. Then we have that $\nu(e_{ij})=l^2_{ij}$ and $\nu(X)=|V|_{l}$.

Assume $u:V\rightarrow\mathbb R$ is linearly extended to the 1-skeleton $X$, and then by maximum principle, $u_a=\min(u)$ and $u_b=\max(u)$.
Let $\bar u\in(u_a,u_b)$ be such that
	\begin{equation*}
	\label{halfmeasure}
	\begin{array}{c}
	\nu(x\in X:u(x)<\bar u)\leq |V|_l/2,\\
	\\
	\nu(x\in X:u(x)>\bar u)\leq |V|_l/2.
	\end{array}
	\end{equation*}

Let $f(x)=l_{ij}|u_i-u_j|$ for $x\in e_{ij}$, and then $f$ is well-defined almost everywhere on $X$, and
	$$
	\sum_{ij\in E}l_{ij}^2|u_i-u_j|=\int_{X}f(x)d\mu\leq
\int_{u_a\leq u(x)\leq \bar u}f(x)d\mu
+
\int_{\bar u\leq u(x)\leq u_b}f(x)d\mu.
	$$
We will prove
	$$
\int_{u_a\leq u(x)\leq \bar u}f(x)d\mu\leq\frac{2\sqrt{C_1+1}}{C_2}|l|\cdot\sqrt{\nu(u(x)< \bar u)}
\leq\frac{2\sqrt{C_1+1}}{C_2}|l|\cdot|V|_l^{1/2}
	$$
and then by the symmetry $\int_{\bar u\leq u(x)\leq u_b}f(x)d\mu$ has the same upper bound and we are done.

	Let $u_a=p_0< p_1<\cdots< p_s=\bar u$ such that $\{p_0\cdots p_{s-1}\}=\{u_i:i\in V,u_i<\bar u\}$.
Noticing that $\int_{u(x)=p}f(x)d\mu=0$ for any $p\in\mathbb R$,
	it suffices to prove that for any $k\in\{1,...,s\}$
	$$
	\int_{p_{k-1}<u(x)<p_k}f(x)d\mu\leq \frac{2\sqrt{C_1+1}}{C_2}|l|\cdot\bigg({\sqrt{\nu(u(x)<p_k)}}-\sqrt{\nu(u(x)<p_{k-1})}\bigg).
	$$
In the remaining of the proof, we fix a $k\in\{1,...,s\}$ and
let $V_k=\{i\in V: u(i)\leq p_{k-1}\}$. Then for any $i\in V_k$ and $ij\in \partial V_k$, $u_j\geq u_i$ and then
	\begin{align*}
	0&=\sum\limits_{i\in V_k-\{a\}}(\Delta u)_{i}=\sum\limits_{i\in V_k-\{a\}}\sum\limits_{j\sim i}\eta_{ij}(u_j-u_i)\\
	&=\sum\limits_{i:i\sim a}\eta_{ia}(u_a-u_i)+\sum\limits_{ij\in \partial V_k:i\in V_k}\eta_{ij}(u_j-u_i)\\
	&\geq-(\Delta u)_a+C_2\sum_{ij\in\partial V_k}|u_j-u_i|.
	\end{align*}
	Therefore,
	\begin{equation}
	\label{boudary estimate}
	C_2\sum_{ij\in\partial V_k}|u_j-u_i|\leq(\Delta u)_a=1.
	\end{equation}

	Let $e_{ij}'=\{x:p_{k-1}<u(x)<p_k\}\cap e_{ij}$, and $l'_{ij}=\mu(e_{ij}')$. Then  $l_{ij}'=0$ if
$ij\notin\partial V_{k}$. If $ij\in \partial V_k$, let $i'$ and $j'$ be the two endpoints of $e_{ij}'$. Then $\{u(i'),u(j')\}= \{p_{k-1},p_k\}$, and
	$$
	\frac{l_{ij}'}{l_{ij}}=\frac{|u(j')-u(i')|}{|u_j-u_i|}= \frac{p_{k}-p_{k-1}}{|u_j-u_i|}.
	$$
	So
	\begin{equation}
\label{30}
	\int_{p_{k-1}<u(x)<p_k}f(x)d\mu\\
	=\sum_{ij\in \partial V_k}l_{ij}'l_{ij}|u_i-u_j|\\
	=(p_k-p_{k-1}) \sum_{ij\in \partial V_k}l_{ij}^2\\
	\leq(p_k-p_{k-1})|l|\cdot|\partial V_k|_l.
	\end{equation}
	On the other hand, by inequality (\ref{boudary estimate}) and Cauchy's inequality,
\begin{align}
\begin{split}
\label{31}
	&\nu(p_{k-1}<u(x)<p_k)=\sum_{ij\in \partial V_k}l_{ij}'l_{ij}=
	(p_k-p_{k-1})\sum_{ij\in \partial V_k}\frac{l_{ij}^2}{|u_j-u_i|}\\
	\geq&(p_k-p_{k-1})\left(\sum_{ij\in \partial V_k}\frac{l_{ij}^2}{|u_j-u_i|}\right)\cdot\left(\sum_{ij\in\partial V_k}|u_j-u_i|\right)\cdot C_2\\
\geq& C_2(p_k-p_{k-1}){(\sum_{ij\in{\partial V_k}}{l_{ij}})^2}=C_2(p_k-p_{k-1})|\partial V_k|_l^2.
\end{split}
\end{align}
	Since $(G,l)$ is $C_1$-isoperimetric, we have that
\begin{equation}
\label{32}
	\nu(u(x)<p_k)\leq|V_k|_l+\sum_{ij\in\partial V_k}l_{ij}^2
	\leq C_1|\partial V_k|_l^2+|\partial V_k|_l^2
	=(C_1+1)|\partial V_k|_l^2.
\end{equation}

Divide(\ref{31}) by $\sqrt{(\ref{32})}$ and then we have
\begin{equation}
\label{33combine}
\frac{\nu(p_{k-1}<u(x)<p_k)}{\sqrt{\nu(u(x)<p_k)}}\geq \frac{C_2}{\sqrt{C_1+1}}(p_k-p_{k-1})|\partial V_k|_l.
\end{equation}
Combining equations (\ref{30}) and (\ref{33combine}) and then
\begin{align*}
	&\int_{p_{k-1}<u(x)<p_k}f(x)d\mu	
\leq\frac{\sqrt{C_1+1}}{C_2}\cdot|l|\cdot\frac{\nu(p_{k-1}<u(x)<p_k)}{\sqrt{\nu(u(x)<p_k)}}\\
	\leq&\frac{\sqrt{C_1+1}}{C_2}\cdot|l|\cdot \frac{\nu(u(x)<p_k)-\nu(u(x)< p_{k-1})}{\sqrt{\nu(u(x)<p_k)}}\\
	\leq&\frac{2\sqrt{C_1+1}}{C_2}\cdot|l|\cdot \frac{\nu(u(x)<p_k)-\nu(u(x)< p_{k-1})}{\sqrt{\nu(u(x)<p_k)}+\sqrt{\nu(u(x)<p_{k-1})}}\\
	=&\frac{2\sqrt{C_1+1}}{C_2}\cdot|l|\cdot\bigg({\sqrt{\nu(u(x)<p_k)}}-\sqrt{\nu(u(x)<p_{k-1})}\bigg)
\end{align*}
and we are done.
\end{proof}

\begin{appendices}
In this appendix we prove Lemma \ref{Euclidean infinitesimal} and \ref{hyperbolic infinitesimal} and \ref{estimate for Euclidean triangle} and \ref{estimate for hyperbolic triangle}.
\begin{manualtheorem}{Lemma 3.3}
Given a Euclidean triangle $\triangle ABC$, if we view $A,B,C$ as functions of the edge lengths $a,b,c$, then
 \begin{equation*}
 \label{differential for Euclidean triangle}
 \frac{\partial A}{\partial b}=-\frac{\cot C}{b},\quad\quad
 \frac{\partial A}{\partial a}=\frac{\cot B+\cot C}{a}=\frac{1}{b\sin C}.
 \end{equation*}
Further if $(u_A,u_B,u_C)\in\mathbb R^3$ is a discrete conformal factor, and
\begin{equation*}
a=e^{\frac{1}{2}(u_B+u_C)}a_0,\quad b=e^{\frac{1}{2}(u_A+u_C)}b_0,
\quad c=e^{\frac{1}{2}(u_A+u_B)}c_0
\end{equation*}
for some constants $a_0,b_0,c_0\in\mathbb R_{>0}$, then
\begin{equation}
\label{differential for Euclidean conformal factor}
\frac{\partial A}{\partial u_B}=\frac{1}{2}\cot C,\quad
\frac{\partial A}{\partial u_A}=-\frac{1}{2}(\cot B+\cot C).
\end{equation}
\end{manualtheorem}
\begin{proof}
Take the partial derivative  on
$$
\cos A=\frac{b^2+c^2-a^2}{2bc}
$$
and we have
$$
-\sin A\frac{\partial A}{\partial b}=\frac{2b}{2bc}-\frac{b^2+c^2-a^2}{2b^2c}=\frac{b^2+a^2-c^2}{2b^2c}=\frac{a\cos C}{bc},
$$
and
$$
\frac{\partial A}{\partial b}=-\frac{a\cos C}{bc\sin A}=-\frac{\cos C}{b\sin C}=-\frac{\cot C}{b}.
$$
Similarly
$$
-\sin A\frac{\partial A}{\partial a}=-\frac{a}{bc}
$$
and
$$
\frac{\partial A}{\partial a}=\frac{a}{bc\sin A}=\frac{1}{b\sin C}=\frac{\sin A}{a\sin B\sin C}
=\frac{\sin B\cos C+\sin C\cos B}{a\sin B\sin C}=\frac{\cot B+\cot C}{a}.
$$
Then equation (\ref{differential for Euclidean conformal factor}) can be computed easily.
\end{proof}

\begin{manualtheorem}{Lemma 3.4}
Given a hyperbolic triangle $\triangle ABC$, if we view $A,B,C$ as functions of the edge lengths $a,b,c$, then
 \begin{equation*}
 \label{differential for hyperbolic triangle}
 \frac{\partial A}{\partial b}=-\frac{\cot C}{\sinh b},\quad\quad
 \frac{\partial A}{\partial a}=\frac{1}{\sinh b\sin C}.
 \end{equation*}

Further if $(u_A,u_B,u_C)\in\mathbb R^3$ is a discrete conformal factor, and
\begin{equation*}
\sinh\frac{a}{2}=e^{\frac{1}{2}(u_B+u_C)}\sinh\frac{a_0}{2},\quad \sinh\frac{b}{2}=e^{\frac{1}{2}(u_A+u_C)}\sinh\frac{b_0}{2},
\quad \sinh\frac{c}{2}=e^{\frac{1}{2}(u_A+u_B)}\sinh\frac{c_0}{2}
\end{equation*}
for some constants $a_0,b_0,c_0\in\mathbb R_{>0}$, then
\begin{equation}
\label{differential for hyperbolic conformal factor }
\frac{\partial A}{\partial u_B}=\frac{1}{2}\cot \tilde C(1-\tanh^2\frac{c}{2}),
\end{equation}
and
\begin{equation}
\label{differential for hyperbolic conformal factor of same index}
\frac{\partial A}{\partial u_A}
=-\frac{1}{2}\cot \tilde B(1+\tanh^2\frac{b}{2})-\frac{1}{2}\cot \tilde C(1+\tanh^2\frac{c}{2}),
\end{equation}
where
$
\tilde B=\frac{1}{2}(\pi+B-A-C)
$
and
$
\tilde C=\frac{1}{2}(\pi+C-A-B)/2.
$

\end{manualtheorem}
\begin{proof}
Take the partial derivative on
$$
\cos A=\frac{\cosh b\cosh c-\cosh a}{\sinh b\sinh c}
$$
and we have
\begin{align*}
-\sin A\frac{\partial A}{\partial b}=
&\frac{\sinh b\cosh c}{\sinh b\sinh c}-
\frac{\cosh^2 b\cosh c-\cosh a\cosh b}{\sinh^2 b\sinh c}\\
=&\frac{\cosh a\cosh b-\cosh c}{\sinh^2 b\sinh c}=\frac{\sinh a}{\sinh b\sinh c}\cos C,
\end{align*}
and then by the hyperbolic law of sines,
$$
\frac{\partial A}{\partial b}=-\frac{\cos C}{\sinh b\sin C}=-\frac{\cot C}{\sinh b}.
$$
Similarly
$$
-\sin A\frac{\partial A}{\partial a}=-\frac{\sinh a}{\sinh b\sinh c}
$$
and then again by the hyperbolic law of sines
$$
\frac{\partial A}{\partial a}=\frac{1}{\sinh b\sinh C}.
$$
To prove (\ref{differential for hyperbolic conformal factor }) and (\ref{differential for hyperbolic conformal factor of same index}) we need to compute
$$
\frac{\partial c}{\partial u_A}=
\frac{\partial\sinh(c/2)}{\partial u_A}\bigg/\frac{\partial\sinh(c/2)}{\partial c}
=\frac{1}{2}\sinh\frac{c}{2}\bigg/\left(\frac{1}{2}\cosh\frac{c}{2}\right)=\tanh\frac{c}{2},
$$
and other similar formulae hold.

		Since
		$$
		\tanh\frac{x}{2}=\frac{\sinh\frac{x}{2}}{\cosh\frac{x}{2}}=\frac{\sinh\frac{x}{2}\cosh\frac{x}{2}}{\cosh^2\frac{x}{2}}=\frac{\sinh x}{\cosh x+1}
		$$
		and
		\begin{equation}
\label{37}
		\cosh b+1=\frac{\cos A\cos C+\cos B}{\sin A\sin C}+1=\frac{\cos(A-C)+\cos B}{\sin A\sin C},
		\end{equation}
		we have
		\begin{align*}
		&\frac{\tanh\frac{b}{2}}{\tanh\frac{c}{2}}=\frac{\sinh b}{\sinh c}\cdot\frac{\cosh c+1}{\cosh b+1}\\
		=&\frac{\sin B}{\sin C}\cdot\frac{\cos (A-B)+\cos C}{\sin A\sin B}\cdot\frac{\sin A\sin C}{\cos(A-C)+\cos B}\\
		=&\frac{\cos (A-B)+\cos C}{\cos(A-C)+\cos B},
		\end{align*}
		and then
		\begin{align*}
		-\cos A+\frac{\tanh \frac{b}{2}}{\tanh\frac{c}{2}}=&\frac{\cos (A-B)+\cos C-\cos(A-C)\cos A-\cos B\cos A}{\cos(A-C)+\cos B}\\
		=&\frac{(\cos(A-B)-\cos B\cos A)+(\cos C-\cos(A-C)\cos A)}{\cos(A-C)+\cos B}\\
		=&\frac{\sin A\sin B+\sin A\sin(A-C)}{\cos(A-C)+\cos B}\\
		=&\sin A\cdot\frac{\sin B+\sin (A-C)}{\cos(A-C)+\cos B}\\
		=&\sin A\cdot\frac{\sin\frac{B+A-C}{2}\cos\frac{B-A+C}{2}}{\cos\frac{A+B-C}{2}\cos\frac{A-B-C}{2}}\\
		=&\sin A\cdot\tan\frac{A+B-C}{2}\\
		=&\sin A\cot\tilde C,
		\end{align*}
and then
		\begin{align*}
		\frac{\partial B}{\partial u_A}=&\frac{\partial B}{\partial c}\frac{\partial c}{\partial u_A}+\frac{\partial B}{\partial b}\frac{\partial b}{\partial u_A}\\
		=&-\frac{\cot A}{\sinh c}\tanh\frac{c}{2}+\frac{1}{\sin A\sinh c}\tanh\frac{b}{2}\\
		=&\frac{1}{2\cosh^2\frac{c}{2}}(-\frac{\cos A}{\sin A}+\frac{1}{\sin A}\frac{\tanh\frac{b}{2}}{\tanh\frac{c}{2}})\\
		=&\frac{1}{2}(1-\tanh^2\frac{c}{2})\frac{1}{\sin A}(-\cos A+\frac{\tanh\frac{b}{2}}{\tanh\frac{c}{2}})\\
		=&\frac{1}{2}(1-\tanh^2\frac{c}{2})\cot\tilde C.
		\end{align*}
By the symmetry equation (\ref{differential for hyperbolic conformal factor }) is true, and for the equation (\ref{differential for hyperbolic conformal factor of same index}), we have that
		$$
		\frac{\partial A}{\partial u_A}=\frac{\partial A}{\partial c}\frac{\partial c}{\partial u_A}+\frac{\partial A}{\partial b}\frac{\partial b}{\partial u_A}=-\frac{\cot B}{\sinh c}\tanh\frac{c}{2}-\frac{\cot C}{\sinh b}\tanh\frac{b}{2}.
		$$
		So we need to show
		$$
		-\frac{\cot B}{\sinh c}\tanh\frac{c}{2}-\frac{\cot C}{\sinh b}\tanh\frac{b}{2}=
		-\frac{1}{2}\cot\tilde C(1+\tanh^2\frac{c}{2})-\frac{1}{2}\cot\tilde B(1+\tanh^2\frac{b}{2}).
		$$
		Since
		$$
		\frac{\tanh\frac{x}{2}}{\sinh x}=\frac{\sinh\frac{x}{2}}{2\sinh\frac{x}{2}\cosh^2\frac{x}{2}}=\frac{1}{\cosh^2\frac{x}{2}}=\frac{2}{\cosh x+1}
		$$
		and
		$$
		1+\tanh^2 \frac{x}{2}=\frac{\cosh^2\frac{x}{2}+\sinh^2\frac{x}{2}}{\cosh^2\frac{x}{2}}=\frac{2\cosh x}{\cosh x+1},
		$$
		we only need to show
		$$
		\frac{\cot B}{\cosh c+1}+\frac{\cot C}{\cosh b+1}=\cot\tilde C\frac{\cosh c}{\cosh c+1}+\cot\tilde B\frac{\cosh b}{\cosh b+1}.
		$$
		We will show that
		$$
		\cot\tilde B\frac{\cosh b}{\cosh b+1}-\frac{\cot B}{\cosh c+1}
		$$
		is anti-symmetric with respect to $B$ and $C$.
		Recall equation (\ref{37}) and we have that
		$$
		\cosh b+1=\frac{\cos( A-C)+\cos B}{\sin A\sin C}=\frac{2\cos\frac{A+B-C}{2}\cos\frac{B+C-A}{2}}{\sin A\sin C},
		$$
		and
		$$
		\frac{\cosh b}{\cosh b+1}=\frac{\cos A\cos C+\cos B}{\cos( A-C)+\cos B}=\frac{\cos A\cos C+\cos B}{2\cos\frac{A+B-C}{2}\cos\frac{B+C-A}{2}}
		$$
		and
		$$
		\cot\tilde B=\tan(\frac{\pi}{2}-\tilde B)=\tan\frac{A+C-B}{2}.
		$$
		So
		\begin{align*}
		&\cot\tilde B\frac{\cosh b}{\cosh b+1}-\frac{\cot B}{\cosh c+1}\\
		=&\tan\frac{A+C-B}{2}\cdot
		\frac{\cos A\cos C+\cos B}{2\cos\frac{A+B-C}{2}\cos\frac{B+C-A}{2}}
		-\cot B\frac{\sin A\sin B}{2\cos\frac{A+C-B}{2}\cos\frac{B+C-A}{2}}\\
		=&\frac{\sin\frac{A+C-B}{2}(\cos A\cos C+\cos B)-\sin A\cos B\cos\frac{A+B-C}{2}}
		{2\cos\frac{A+C-B}{2}\cos\frac{B+C-A}{2}\cos\frac{A+B-C}{2}}.
		\end{align*}
		The denominator in the above fraction is symmetric, so we only need to show the numerator is anti-symmetric with respect to $B,C$.
		\begin{align*}
		&\sin\frac{A+C-B}{2}(\cos A\cos C+\cos B)-\sin A\cos B\cos\frac{A+B-C}{2}\\
		=&(\sin\frac{A}{2}\cos\frac{C-B}{2}+\sin\frac{C-B}{2}\cos\frac{A}{2})(\cos A\cos C+\cos B)\\
		&-\sin A\cos B(\cos\frac{A}{2}\cos\frac{C-B}{2}+\sin\frac{A}{2}\sin\frac{C-B}{2})\\
		=&\sin\frac{C-B}{2}(\cos\frac{A}{2}\cos A\cos C+\cos\frac{A}{2}\cos B-\sin A\cos B\sin\frac{A}{2})\\
		&+\cos\frac{C-B}{2}(\sin\frac{A}{2}\cos A\cos C+\sin\frac{A}{2}\cos B-\sin A\cos B\cos\frac{A}{2})\\
		=&\sin\frac{C-B}{2}(\cos\frac{A}{2}\cos A\cos C+\cos A\cos B\cos\frac{A}{2})\\
		&+\cos\frac{C-B}{2}(\sin\frac{A}{2}\cos A\cos C-\cos A\cos B\sin\frac{A}{2})\\
		=&\sin\frac{C-B}{2}\cos A\cos\frac{A}{2}(\cos C+\cos B)+\cos\frac{C-B}{2}\sin\frac{A}{2}\cos A(\cos C-\cos B)
		\end{align*}
		is indeed anti-symmetric with respect to $B,C$.
\end{proof}

\begin{manualtheorem}{Lemma 3.5}
Given a Euclidean triangle $\triangle ABC$, if all the angles in $\triangle ABC$ are at least $\epsilon>0$, and $\delta<\epsilon^2/48$, and
$$
|a'-a|\leq \delta a,\quad |b'-b|\leq \delta a,\quad |c'-c|\leq \delta c,
$$
then $a',b',c'$ form a Euclidean triangle with opposite inner angles $A',B',C'$ respectively, and
$$
|A'-A|\leq\frac{24}{\epsilon}\delta,
$$
and
$$
\bigg||\triangle A'B'C'|-|\triangle ABC|\bigg|\leq \frac{576}{\epsilon^2}\delta\cdot |\triangle ABC|.
$$
\end{manualtheorem}
\begin{proof}

Let
$$
u_A(t)=t\cdot(\log\frac{b'}{b}+\log\frac{c'}{c}-\log\frac{a'}{a})
$$
and $u_B(t),u_C(t)$ be defined similarly.
Then $|u'|\leq-3\log(1-\delta)\leq 6\delta$, since
$$
\delta\leq\frac{\epsilon^2}{48}\leq\frac{(\pi/3)^2}{48}\leq0.1.
$$
Assume
$$
a(t)=e^{\frac{1}{2}(u_B(t)+u_C(t))}a,\quad b(t)=e^{\frac{1}{2}(u_A(t)+u_C(t))}b,\quad c(t)=e^{\frac{1}{2}(u_A(t)+u_B(t))}c,
$$
and then $a(1)=a',b(1)=b',c(1)=c'$.
Let $A(t),B(t),C(t)$ be the inner angles of the triangle with edge lengths $a(t),b(t),c(t)$, if well-defined.

Let $T_0\in[0,\infty]$ be the maximum real number such that for any $t\in[0,T_0)$, all $A(t),B(t),C(t)>\epsilon/2$.
Then $T_0>0$ and for any $t\in[0,T_0)$, by Lemma \ref{Euclidean infinitesimal}
$$
|A'(t)|=
\left|\frac{\partial A}{\partial u_A}u_A'+\frac{\partial A}{\partial u_B}u_B'+\frac{\partial A}{\partial u_C}u_C'\right|
\leq 2\cot\frac{\epsilon}{2}\cdot|u'|\leq12\delta\cot\frac{\epsilon}{2}\leq \frac{24}{\epsilon}\delta,
$$
and similarly $|B'(t)|,|C'(t)|\leq 24\delta/\epsilon$.
So $T_0\geq(\epsilon/2)/(24\delta/\epsilon)=\epsilon^2/(48\delta)>1$, and $|A'-A|\leq 24\delta/\epsilon$.

By Lemma \ref{Euclidean infinitesimal} for $t\in(0,1)$
$$
\frac{\partial |\triangle ABC|}{\partial a}=\frac{\partial (\frac{1}{2}bc\sin A)}{\partial A}\cdot\frac{\partial A}{\partial a}
=\frac{1}{2}bc\cos A\cdot\frac{a}{bc\sin A}=\frac{a(t)}{2\tan A(t)}
$$
and then by the chain rule
$$
\left|\frac{d|\triangle ABC(t)|}{dt}\right|\leq
|u'|\cdot\left(\left|\frac{a^2}{2\tan A}\right|+\left|\frac{b^2}{2\tan B}\right|+\left|\frac{c^2}{2\tan C}\right|\right)\leq6\delta\cdot\frac{a(t)^2+b(t)^2+c(t)^2}{\epsilon},
$$
where $a(t)\leq e^{|u(t)|}a\leq e^{6\delta t}a\leq2a$  and $b(t)\leq2b$ and $c(t)\leq 2c$.

Then by Lemma \ref{triangle area},
$$
\big||\triangle A'B'C'|-|\triangle ABC|\big|\leq\frac{24\delta}{\epsilon}(a^2+b^2+c^2)\leq
\frac{24\delta}{\epsilon}\cdot 3\cdot\frac{8}{\epsilon}\cdot|\triangle ABC|
=\frac{576}{\epsilon^2}\delta|\triangle ABC|.
$$
\end{proof}

\begin{manualtheorem}{Lemma 3.6}
Given a hyperbolic triangle $\triangle ABC$, if all the angles in $\triangle ABC$ are at least $\epsilon>0$, and $\delta<\epsilon^3/60$, and
$$
a\leq0.1,\quad b\leq0.1
,\quad c\leq0.1,
$$
and
$$
|a'-a|\leq \delta a,\quad |b'-b|\leq \delta a,\quad |c'-c|\leq \delta c,
$$
then $a',b',c'$ form a hyperbolic triangle with opposite inner angles $A',B',C'$ respectively, and
$$
|A'-A|\leq\frac{30}{\epsilon^2}\delta,
$$
and
$$
\bigg||\triangle A'B'C'|-|\triangle ABC|\bigg|\leq \frac{120}{\epsilon^2}\delta \cdot |\triangle ABC|.
$$
\end{manualtheorem}

\begin{proof}

Let
$$
a(t)=ta'+(1-t)a,\quad b(t)=tb'+(1-t)b,\quad c(t)=tc'+(1-t)c,
$$
and
$A(t),B(t),C(t)$ be the inner angles of the triangle with edge lengths $a(t),b(t),c(t)$, if well-defined.

Let $T_0\in[0,\infty]$ be the maximum real number such that for any $t\in[0,T_0)$, all $A(t),B(t),C(t)>\epsilon/2$.
Notice that $\delta<\epsilon^3/60<0.1$ and then for any $t\in[0,T_0)$, $\sinh a(t)\in[a,2a]$ and so on.
By Lemma \ref{hyperbolic infinitesimal}
\begin{align*}
&|\dot A(t)|=
\left|\frac{\partial A}{\partial a}\dot a+\frac{\partial A}{\partial b}\dot b+\frac{\partial A}{\partial c}\dot c\right|\\
\leq &\frac{|a'-a|}{\sinh b(t)\sin(\epsilon/2)}+\frac{\cot(\epsilon/2)|b'-b|}{\sinh b(t)}+\frac{\cot(\epsilon/2)|c'-c|}{\sinh c(t)}\\
\leq&\frac{|a'-a|}{\sinh a(t)\sin^2(\epsilon/2)}+\frac{\cot(\epsilon/2)|b'-b|}{\sinh b(t)}+\frac{\cot(\epsilon/2)|c'-c|}{\sinh c(t)}\\
\leq&2\delta\left(\frac{1}{\sin^2(\epsilon/2)}+2\cot(\epsilon/2)\right)\\
\leq&2(\frac{\pi^2}{\epsilon^2}+\frac{4}{\epsilon})\delta
\leq\frac{30}{\epsilon^2}\delta.
\end{align*}

and similarly $|\dot B(t)|,|\dot C(t)|\leq 30\delta/\epsilon^2$.
So $T\geq(\epsilon/2)/(30\delta/\epsilon^2)=\epsilon^3/(60\delta)>1$, and
$$
|A'-A|\leq 30\delta/\epsilon^2,\quad|B'-B|\leq 30\delta/\epsilon^2,\quad|C'-C|\leq 30\delta/\epsilon^2.
$$
For $t\in[0,1]$, by Lemma \ref{hyperbolic infinitesimal}
\begin{align*}
&\left|\frac{\partial(A+B+C)}{\partial a}\right|=\left|\frac{1}{\sinh b\sin C}-\frac{\cot C}{\sinh a}-\frac{\cot B}{\sinh a}\right|
=\left|\frac{1}{\sinh a}\frac{\sin A-\sin(B+C)}{\sin B\sin C}\right|\\
\leq&\frac{1}{\sinh a}\frac{|\sin(\pi-A)-\sin(B+C)|}{\sin^2(\epsilon/2)}\leq
\frac{\pi^2(\pi-A(t)-B(t)-C(t))}{\epsilon^2\sinh a(t)}\leq\frac{2\pi^2}{\epsilon^2}\frac{|\triangle ABC(t)|}{a},
\end{align*}
and then
\begin{align*}
&\left|\frac{d|\triangle ABC|}{dt}\right|=
|\dot A+\dot B+\dot C|
\leq\frac{2\pi^2}{\epsilon^2}|\triangle ABC(t)|\left(\frac{|a'-a|}{a}+\frac{|b'-b|}{b}+\frac{|c'-c|}{c}
\right)\\
\leq&
\frac{6\pi^2\delta}{\epsilon^2}|\triangle ABC(t)|.
\end{align*}
So
$$
\frac{|\triangle A'B'C'|}{|\triangle ABC|}\in[ e^{-6\pi^2\delta/\epsilon^2},e^{6\pi^2\delta/\epsilon^2}]\in[1-6\pi^2\delta/\epsilon^2,1+120\delta/\epsilon^2]
$$
and
$$
\big||\triangle A'B'C'|-|\triangle ABC|\big|\leq \frac{120}{\epsilon^2}\cdot\delta\cdot|\triangle ABC|.
$$
\end{proof}

\end{appendices}

\bibliographystyle{unsrt}
\bibliography{kc2gkc2}

\begin{thebibliography}{10}

\bibitem{koehl2013automatic}
Patrice Koehl and Joel Hass.
\newblock Automatic alignment of genus-zero surfaces.
\newblock {\em IEEE transactions on pattern analysis and machine intelligence},
  36(3):466--478, 2013.

\bibitem{gu2004genus}
Xianfeng Gu, Yalin Wang, Tony~F Chan, Paul~M Thompson, and Shing-Tung Yau.
\newblock Genus zero surface conformal mapping and its application to brain
  surface mapping.
\newblock {\em IEEE transactions on medical imaging}, 23(8):949--958, 2004.

\bibitem{boyer2011algorithms}
Doug~M Boyer, Yaron Lipman, Elizabeth~St Clair, Jesus Puente, Biren~A Patel,
  Thomas Funkhouser, Jukka Jernvall, and Ingrid Daubechies.
\newblock Algorithms to automatically quantify the geometric similarity of
  anatomical surfaces.
\newblock {\em Proceedings of the National Academy of Sciences},
  108(45):18221--18226, 2011.

\bibitem{hong2006conformal}
Wei Hong, Xianfeng Gu, Feng Qiu, Miao Jin, and Arie Kaufman.
\newblock Conformal virtual colon flattening.
\newblock In {\em Proceedings of the 2006 ACM symposium on Solid and physical
  modeling}, pages 85--93, 2006.

\bibitem{lui2010optimized}
Lok~Ming Lui, Sheshadri Thiruvenkadam, Yalin Wang, Paul~M Thompson, and Tony~F
  Chan.
\newblock Optimized conformal surface registration with shape-based landmark
  matching.
\newblock {\em SIAM Journal on Imaging Sciences}, 3(1):52--78, 2010.

\bibitem{luo2004combinatorial}
Feng Luo.
\newblock Combinatorial yamabe flow on surfaces.
\newblock {\em Communications in Contemporary Mathematics}, 6(05):765--780,
  2004.

\bibitem{bobenko2015discrete}
Alexander~I Bobenko, Ulrich Pinkall, and Boris~A Springborn.
\newblock Discrete conformal maps and ideal hyperbolic polyhedra.
\newblock {\em Geometry \& Topology}, 19(4):2155--2215, 2015.

\bibitem{gu2019convergence}
David Gu, Feng Luo, and Tianqi Wu.
\newblock Convergence of discrete conformal geometry and computation of
  uniformization maps.
\newblock {\em Asian Journal of Mathematics}, 23(1):21--34, 2019.

\bibitem{mednykh2012brahmagupta}
AD~MEDNYKH.
\newblock Brahmagupta formula for cyclic quadrilaterals in the hyperbolic
  plane.
\newblock {\em Sibirskie Elektronnye Matematicheskie Izvestiia}, 9, 2012.

\bibitem{zwillinger2002crc}
Daniel Zwillinger.
\newblock {\em CRC standard mathematical tables and formulae}.
\newblock CRC press, 2002.

\bibitem{frenkel2018area}
Elena Frenkel.
\newblock {\em On area and volume in spherical and hyperbolic geometry}.
\newblock PhD thesis, 2018.

\bibitem{petersen2006riemannian}
Peter Petersen.
\newblock {\em Riemannian geometry}, volume 171.
\newblock Springer, 2006.

\bibitem{alexander1993geometric}
Stephanie~B Alexander, I~David Berg, and Richard~L Bishop.
\newblock Geometric curvature bounds in riemannian manifolds with boundary.
\newblock {\em Transactions of the American Mathematical Society},
  339(2):703--716, 1993.

\bibitem{bridson1999metric}
Martin~R Bridson and Andr{\'e} Haefliger.
\newblock Metric spaces of non-positive curvature.
\newblock 1999.

\bibitem{gray2006modern}
Alfred Gray, Elsa Abbena, S~Salamon, et~al.
\newblock Modern differential geometry of curves and surfaces with mathematica.
\newblock 2006.

\bibitem{osserman1978isoperimetric}
Robert Osserman et~al.
\newblock The isoperimetric inequality.
\newblock {\em Bulletin of the American Mathematical Society},
  84(6):1182--1238, 1978.

\end{thebibliography}

\end{document}